\documentclass[a4paper,11pt]{article}

\usepackage{graphicx}
\usepackage{amsfonts, amsthm, amssymb, amsmath}
\usepackage[pass]{geometry}
\usepackage{setspace}
\usepackage[english]{babel}
\usepackage{makeidx}
\usepackage{tikz-cd}
\usepackage[utf8]{inputenc}
\usepackage{mathtools}
\usepackage{enumitem}
\usepackage{mathrsfs,authblk}
\usepackage{xcolor}
\usepackage{hyperref}

\hypersetup{
colorlinks=true,
linkcolor=cyan,
filecolor=magenta,
urlcolor=cyan,
}

\newtheorem{theorem}{Theorem}[section]
\newtheorem{lemma}[theorem]{Lemma}
\newtheorem{proposition}[theorem]{Proposition}

\newtheorem{definition}[theorem]{Definition}
\theoremstyle{remark}
\newtheorem{remark}[theorem]{Remark}

\newtheorem{problem}{Problem}

\numberwithin{equation}{section}

\newcommand{\smallO}[1]{\scriptstyle\mathcal{O}}

\newcommand{\restr}{\mathbin{\vrule height 1.6ex depth 0pt width 0.13ex\vrule height 0.13ex depth 0pt width 1.3ex}}
\newcommand{\bib}[4]{\bibitem{#1}{\sc#2: }{\it#3. }{#4.}}

\newcommand{\R}{\mathbb R}

\newcommand{\Om}{\Omega}

\newcommand{\cA}{\mathcal A}
\newcommand{\cE}{\mathcal E}
\newcommand{\cH}{\mathcal H}

\newcommand{\eps}{\varepsilon}
\newcommand{\diam}{\mathrm{diam}}
\newcommand{\divs}{\mathrm{div}}
\newcommand{\spt}{\mathrm{spt}}
\newcommand{\sA}{\mathscr A}
\newcommand{\cL}{\mathcal L}
\newcommand{\be}{\begin{equation}}
\newcommand{\ee}{\end{equation}}
\newcommand{\can}{\symbol{35}}

\title{Optimal one-dimensional structures for the principal eigenvalue of two-dimensional domains}

\author{Giuseppe Buttazzo and Francesco Paolo Maiale}

\begin{document}

\maketitle

\begin{abstract}
A shape optimization problem arising from the optimal reinforcement of a membrane by means of one-dimensional stiffeners or from the fastest cooling of a two-dimensional object by means of ``conducting wires'' is considered. The criterion we consider is the maximization of the first eigenvalue and the admissible classes of choices are the one of one-dimensional sets with prescribed total length, or the one where the constraint of being connected (or with an a priori bounded number of connected components) is added. The corresponding relaxed problems and the related existence results are described.
\end{abstract}

\bigskip{\bf Keywords: }Optimal reinforcement, eigenvalues of the Laplacian, stiffeners, fastest cooling.

\medskip{\bf2010 Mathematics Subject Classification: }49J45, 35R35, 35J25, 49Q10.

\section{Introduction}\label{intro}

The problem of finding the vibration modes of an elastic membrane $\Om\subset\R^2$, fixed at its boundary $\partial\Om$, is known to reduce to the PDE
$$\begin{cases}-\Delta u=\lambda u&\text{in }\Om,\\
u=0&\text{on }\partial \Om.
\end{cases}$$
The eigenvalues $\lambda$ for which the PDE above has nonzero solutions are all strictly positive with no finite limit, hence they can be ordered as
$$0<\lambda_1\le\lambda_2\le\dots\le\lambda_k\le\dots \to +\infty.$$
We are interested in the behavior of the first eigenvalue, which can be also characterized via the variational problem
$$\min\left\{\frac{\int_\Om|\nabla u|^2\,dx}{\int_\Om|u|^2\,dx}\ :\ u\in H_0^1(\Om),\ u\ne0\right\}.$$
Our goal is to see how the value $\lambda_1$ above modifies when we attach to the membrane a one-dimensional stiffener $S$, modeled by a one-dimensional rectifiable set $S\subset\Om$. In this case the first eigenvalue depends on $S$ and is given by
\be\label{eq.2.3}
\lambda_1(S)=\inf\left\{\frac{\int_\Om|\nabla u|^2\,dx+m\int_S|\nabla_\tau u|^2\,d\cH^1}{\int_\Om|u|^2\,dx}\ :\ u \in C_c^\infty(\Om),\ u\ne0\right\},
\ee
where $\nabla_\tau$ is the tangential derivative and the parameter $m$ indicates the stiffness coefficient of the material of which $S$ is made of.

A similar problem arises in the heat diffusion when a two-dimensional heat conductor, with zero temperature at the boundary and initial temperature $u_0$, has to be cooled as fast as possible by adding one-dimensional strongly conducting wires $S$. The corresponding second order operator in presence of the structure $S$ is given in the weak form by
$$\langle\cA_S u,\phi\rangle=\int_\Om\nabla u\nabla\phi\,dx+m\int_S\nabla_\tau u\nabla_\tau\phi\,d\cH^1,$$
where $u$ and $\phi$ vary in the Sobolev space $H^1_0(\Om)\cap H^1(S)$. By the Fourier analysis we may write the solution of the heat equation
$$\begin{cases}
\partial_tu+\cA_S u=0&\text{in }]0,T[\times\Om\\
u=0&\text{on }]0,T[\times\partial\Om\\
u=u_0&\text{on $\Om$ for }t=0
\end{cases}$$
as
$$u(t,x)=\sum_{k\ge1}c_k u_k(x)e^{-t\lambda_k(S)},\qquad c_k=\int_\Om u_0u_k\,dx$$
where $\lambda_k(S)$ are the eigenvalues of the operator $\cA_S$ and $u_k$ the corresponding eigenfunctions (normalized with unitary $L^2$ norm). The fastest cooling then reduces to searching the structure $S$ providing the maximal first eigenvalue among the class of admissible choices for $S$.

In the present paper we consider the shape optimization problem related to the functional $\lambda_1(S)$ defined in \eqref{eq.2.3} on the following two classes of admissible choices for the stiffener $S$, where $\cL(S)$ denotes the length of $S$:
\[\begin{split}
&\cA_L=\big\{S\subset\bar\Om,\ S\text{ rectifiable, }\cL(S)\le L\big\};\\
&\cA^c_L=\big\{S\subset\bar\Om,\ S\text{ rectifiable, }S\text{ connected, }\cL(S)\le L\big\}.
\end{split}\]
Similarly, we could consider the admissible class of stiffeners having at most $N$ connected components
$$\cA^{c,N}_L=\big\{S\subset\bar\Om,\ S\text{ rectifiable, }S\text{ has most $N$ connected components, }\cL(S)\le L\big\}.$$
We do not consider this last situation since there are no essential differences between the cases $N>1$ and $N=1$. For a general presentation of shape optimization problem we refer to the books \cite{bubbut1} and \cite{henpie1}

In Section \ref{sec:1} we give the precise formulation of the two optimization problems involving the admissible classes $\cA_L$ and $\cA^c_L$ and their corresponding relaxed formulations. We will show that the relaxed problems admit a solution, which will be in both cases a measure $\mu$ supported in $\Om$ in the first case and on a rectifiable set $S$ in the second one. Our main results are that these measures do not have singular parts; more precisely, in the case $\cA_L$ it is a function $\theta\in L^p(\Om)$, while in the case $\cA^c_L$ it is a measure of the form $\theta\restr S$ where $S$ is a suitable connected set and $\theta\in L^1(S)$.

Section \ref{sproofs} contains the proofs of the results. In Section \ref{sradial} we consider the case when $\Om$ is a disk, in which some explicit calculations can be made for the relaxed optimization problem related to the choice $\cA_L$ of admissible sets. Section \ref{sec:4} deals with the case $\cA^c_L$ in which the admissible sets $S$ are connected. Finally, in Section \ref{sopen} we collected some open questions that in our opinion merit some further investigation.

\section{Formulation of the problem and main results}\label{sec:1}

Let $\Om\subset\R^2$ be a bounded Lipschitz domain. The two optimization problems we consider are
\be\label{optpb1}
\max\big\{\lambda_1(S)\ :\ S\in\cA_L\big\}
\ee
\be\label{optpb2}
\max\big\{\lambda_1(S)\ :\ S\in\cA^c_L\big\}
\ee
where $\lambda_1(S)$ is defined in \eqref{eq.2.3}. We now deduce in the two cases the corresponding relaxed problems, obtained by means of the possible limits of admissible $S_n$. In the following we use the notation:
\begin{itemize}
\item[]$|A|$ for the (two-dimensional) Lebesgue measure of $A$;
\item[]$\cL(S)$ for the length of a one-dimensional set $S$;
\item[]$\|\mu\|$ for the mass of a measure $\mu$.
\end{itemize}

Let $(S_n)$ be a sequence of admissible stiffeners for problem \eqref{optpb1}; considering the measures $\mu_n=\cH^1\restr S_n$ we have that $\|\mu_n\|\le L$, hence a subsequence (that we still indicate by $\mu_n$) weakly* converges to a suitable measure $\mu$. It is then convenient to define $\lambda_1(\mu)$ for every measure $\mu$ on $\Om$ by setting
$$\lambda_1(\mu)=\inf\left\{\frac{\int_\Om|\nabla u|^2\,dx+m\int|\nabla u|^2\,d\mu}{\int_\Om|u|^2\,dx}\ :\ u\in C^\infty_c(\Om),\ u\ne0\right\}.$$
Note that, in general the infimum above is not attained on $C^\infty_c(\Om)$ and minimizing sequences converge, strongly in $L^2(\Om)$ and weakly in a suitably defined Sobolev space $H^1_\mu$, to solutions of the relaxed problem
$$\min\left\{\frac{\int_\Om|\nabla u|^2\,dx+m\int|\nabla_\mu u|^2\,d\mu}{\int_\Om|u|^2\,dx}\ :\ u\in H^1_0(\Om)\cap H^1_\mu,\ u\ne0\right\}.$$
Here $\nabla_\mu$ represents a kind of {\it tangential gradient} that was defined in \cite{bobuse} for every measure $\mu$. In this way, when $\mu=\cH^1\restr S$ the tangential gradient $
\nabla_\mu$ coincides with the usual tangential gradient to $S$, so that the definition above of $\lambda_1(\mu)$ reduces to $\lambda_1(S)$. The relaxed version of the optimization problem \eqref{optpb1} then reads
\be\label{relpb1}
\max\big\{\lambda_1(\mu)\ :\ \mu\in\sA_L\big\}
\ee
where $\sA_L$ is the class of nonnegative measures $\mu$ on $\Om$ with $\|\mu\|\le L$.

\begin{proposition}\label{exrel1}
The relaxed optimization problem \eqref{relpb1} admits a solution.
\end{proposition}

\begin{proof}
For every fixed $u\in C^\infty_c(\Om)$ the map
$$\mu\mapsto\frac{\int_\Om|\nabla u|^2\,dx+m\int|\nabla u|^2\,d\mu}{\int_\Om|u|^2\,dx}$$
is weakly* continuous. Hence $\lambda_1(\mu)$ is upper semicontinuous for the weak* convergence, being the infimum of continuous functions. The existence result then follows from the fact that, thanks to the bound $\|\mu\|\le L$, the class $\sA_L$ is weakly* compact.
\end{proof}

\begin{remark}
Following the theory developed in \cite{bobuse} concerning variational integrals with respect to a general measure, the expression of $\lambda_1(\mu)$ can be equivalently given as
$$
\lambda_1(\mu)=\inf\left\{\frac{\int_\Om|\nabla u|^2\,dx+m\int|\nabla_\mu u|^2\,d\mu}{\int_\Om|u|^2\,dx}\ :\ u\in H^1_0(\Om)\cap H^1_\mu,\ u\ne0\right\},
$$
where the Sobolev space $H^1_\mu$ and the ``tangential gradient'' $\nabla_\mu$ are suitably defined. We refer the interested reader to \cite{bobuse}, where the precise definitions and all the details are explained. We will see that for our purposes we do not need these fine tools, since we will obtain that optimal measures for problem \eqref{relpb1} are actually $L^p$ functions, for which the tangential gradient reduces to the usual gradient and the Sobolev space $H^1_0(\Om)\cap H^1_\mu$ reduces to the usual Sobolev space $H^1_0(\Om)$.
\end{remark}

Before stating our main result, we introduce a slightly technical assumption which ensures a bound on the $L^\infty$-norm of the gradient on the boundary of $\Om$.

\begin{definition}[External Ball Condition] \label{def:ebc}
A subset $\Om\subset\R^d$ satisfies the uniform external ball condition with radius $\rho>0$ if
$$\forall x_0\in\partial\Om,\ \exists y_0\in\R^d\ :\ B(y_0,\rho)\subset\R^d\setminus\Om \text{ and }x_0\in\partial B(y_0,\rho).$$
\end{definition}

We will always require $\Om$ connected to work with the ``unique'' eigenfunction which is positive on all $\Om$ (see \cite[Theorem 1.2.5]{henrot1}) and with fixed $L^2$-norm. Our main result concerning optimization problem \eqref{relpb1} is below. 

\begin{theorem}\label{thopt1}
Let $\Om$ be a connected subset of $\R^2$ with Lipschitz boundary satisfying the uniform external ball condition. Then the optimization problem \eqref{relpb1} admits a solution of the form $\mu=\theta\,dx$ where $\theta$ is a function belonging to $L^p(\Om)$ for every $p<+\infty$, equal to zero almost everywhere on the set
$$
\big\{x\in\Om\ :\ |\nabla u_\theta|(x)<\|\nabla u_\theta\|_{L^\infty(\Om)}\big\}
$$
and satisfying the identity
\be\label{eq.2.5}
\max_{\mu\in\sA_L}\lambda_1(\mu)=\min_{u\in H_0^1(\Om)\setminus\{0\}}\left\{\frac{\int_\Om|\nabla u|^2\,dx+mL\|\nabla u\|_{L^\infty(\Om)}^2}{\int_\Om|u|^2\,dx} \right\}.
\ee
Furthermore,
\begin{enumerate}
\item[(i)]if $\Om$ is convex we have $\theta\in L^\infty(\Om)$;
\item[(ii)]if $\partial\Om\in C^{2,\alpha}$, then there exists $\beta=\beta(\alpha)\in(0,1)$ such that $\theta\in C^{1,\beta}(\bar{\Om})$;
\end{enumerate}
\end{theorem}

The proof of theorem above is given in Section \ref{sproofs}. We now consider the relaxation of the optimization problem \eqref{optpb2} where the connectedness constraint is imposed. In this case, if $(S_n)$ is sequence in $\cA_L^c$, the limit of (a subsequence of) $\mu_n=\cH^1\restr S_n$ is still a measure $\mu$ supported by a suitable set $S$. Since the sequence $(S_n)$ is compact in the Hausdorff convergence the set $S$ is closed and connected. In addition, thanks to the Go\l ab theorem (see \cite{gol29}, and the books \cite{amti04}, \cite{fal86}), we have $\mu\ge\cH^1\restr S$, hence the set $S$ verifies $\cL(S)\le L$, so that $S\in\cA_L^c$. Then, introducing the class
$$\sA_L^c=\big\{\mu\hbox{ measure on }\Om,\ \|\mu\|\le L,\ \spt\mu=S\hbox{ connected, }\mu\ge\cH^1\restr S\big\},$$
the relaxed version of the optimization problem \eqref{optpb2} reads
\be\label{relpb2}
\max\big\{\lambda_1(\mu)\ :\ \mu\in\sA^c_L\big\}.
\ee

\begin{proposition}\label{exrel2}
The relaxed optimization problem \eqref{relpb2} admits a solution.
\end{proposition}

\begin{proof}
The proof is similar to the one of Proposition \ref{exrel1}. The map $\lambda_1(\mu)$ is upper semicontinuous for the weak* convergence, and the existence result then follows from the compactness, with respect to the weak* convergence, of the class $\sA_L^c$. This is a consequence of the compactness, for the Hausdorff convergence, of the class of closed and connected sets, and of the Go\l ab theorem, which gives the inequality $\mu\ge\cH^1\restr S$ for a weak* limit $\mu$ of a sequence $\mu_n=\cH^1\restr S_n$ with $S_n$ connected and converging to $S$ in the Hausdorff sense.
\end{proof}

Again, the question if the optimization problem \eqref{relpb2} admits as a solution a measure that is actually a function on a set $S$ arises; this would avoid the use of the delicate theory of variational integrals with respect to a general measure and of the related Sobolev spaces. This is indeed the case and our main result concerning optimization problem \eqref{relpb2} is below.

\begin{theorem}\label{thopt2}
The optimization problem \eqref{relpb2} admits a solution of the form $\mu=\theta \, \cH^1\restr S$, where $S$ is a closed connected subset of $\bar\Om$ with $\cL(S)\le L$ and $\theta\in L^1(S)$ with $\theta\ge1$ on $S$.
\end{theorem}

\begin{remark}If we introduce the class of measures
$$\sA_L^{c,N}=\big\{\mu\hbox{ measure on }\Om,\ \|\mu\|\le L,\ \spt\mu=S\hbox{ $N$-connected, }\mu\ge\cH^1\restr S\big\},$$
where $N$-connected means that is has at most $N$ connected components, then the proof of Theorem \ref{thopt2} easily generalize to the maximization problem 
$$
\max\big\{\lambda_1(\mu)\ :\ \mu\in\sA^{c, \, N}_L\big\}.
$$
In particular, there exist a solution of the form $\mu_N = \sum_{j = 1}^\ell \theta_j \,\cH^1\restr S_j$ where the $S_j$ are closed connected subsets of $\bar\Om$ of total length $\le L$. Moreover, $\theta \in L^1(\cup_{1 \le j \le \ell} S_j)$ and $\theta \ge 1$ on its support. 
\end{remark}

\section{Proof of the results}\label{sproofs}

We start to consider the optimization problem \eqref{relpb1}, which is a max-min problem:
$$\sup_{\mu\in\sA_L}\inf_{u\in C^\infty_c(\Om)}\frac{\int_\Om|\nabla u|^2\,dx+m\int_\Om|\nabla u|^2\,d\mu}{\int_\Om|u|^2\,dx}\;.$$
Proposition \ref{exrel1} gives the existence of an optimal relaxed solution, which is a measure $\mu$ on $\Om$ with $\|\mu\|\le L$. Of course, since the cost above is monotone increasing with respect to $\mu$, optimal measures will saturate the constraint, so we will have $\|\mu\|=L$.

The main result of this section asserts that, under mild assumptions on the boundary of $\Om$, the optimal measures $\mu$ are actually of the form $\theta\,dx$, where $\theta$ is a function that solves the optimization problem
\be\label{optpblambda}
\max\big\{\lambda_1(\theta)\ :\ \theta\in\sA_L\big\}
\ee
and $\lambda_1(\theta)$ is defined by
$$\lambda_1(\theta)=\min\left\{\frac{\int_\Om(1+m\theta)|\nabla u|^2\,dx}{\int_\Om|u|^2\,dx}\ :\ u\in H^1_0(\Om)\right\}.$$
Furthermore, we will see that the optimal densities $\theta$ satisfy some higher-integrability properties and, if $\Om$ is convex, belong to $L^\infty(\Om)$.

In order to obtain better properties of the optimal measure $\mu$ provided by the existence result seen in Proposition \ref{exrel1} it is convenient to consider the optimization problem \eqref{relpb1} under the stronger constraint that $\mu=\theta\,dx$ with $\int_\Om\theta^p\,dx\le L^p$ and $p>1$. In other words, we consider the class
$$\sA_{L,p}=\left\{\theta\in L^p(\Om)\ :\ \theta\ge0\text{ and }\int_\Om\theta^p(x)\,dx \le L^p\right\}$$
and the optimization problem
\be\label{optpb1p}
\max\big\{\lambda_1(\theta)\ :\ \theta\in\sA_{L,p}\big\}
\ee
We still have a max-min problem:
$$\sup_{\theta\in\sA_{L,p}}\inf_{u\in H^1_0(\Om)}\frac{\int_\Om(1+m\theta)|\nabla u|^2\,dx}{\int_\Om|u|^2\,dx}\;.$$

\begin{proposition}\label{proposition:3.2}
For every $p>1$ there exist a unique solution $\theta_p$ of the optimization problem \eqref{optpb1p}, given by
$$\theta_p(x)=L\frac{|\nabla u_p|^{2/(p-1)}}{\||\nabla u_p|^{2/(p-1)}\|_{L^p}}\;.$$
where $u_p$ is the unique positive solution with $\|u_p\|_{L^2(\Om)}=1$ of the auxiliary problem
$$\min_{u\in H_0^1(\Om)\setminus\{0\}}\frac{\int_\Om|\nabla u|^2\,dx+mL\||\nabla u|^2\|_{L^q(\Om)}}{\int_\Om|u|^2\,dx}\;,$$
where $q=p/(p-1)$ is the dual exponent of $p$. Furthermore, the function $u_p$ belongs to $L^\infty(\Om)\cap W_0^{1,2q}(\Om)$ and if in addition $\partial\Om\in C^{2,\alpha}$, then there exists $\beta=\beta(\alpha)\in(0,1)$ such that $u_p\in C^{2,\beta}(\bar{\Om})$. In particular, $\theta_p\in C^{1,\beta}$ up to the boundary.
\end{proposition}

\begin{proof}
Let us denote by $E(\theta,u)$ and by $\cE(u)$ the functionals
\[\begin{split}
&E(\theta,u)=\frac{\int_\Om(1+m\theta)|\nabla u|^2\,dx}{\int_\Om|u|^2\,dx}\;,\\
&\cE(u)=\frac{\int_\Om|\nabla u|^2\,dx+mL\||\nabla u|^2\|_{L^q(\Om)}}{\int_\Om|u|^2\,dx}\;.
\end{split}\]
Then problem \eqref{optpb1p} is written as
$$\max_{\theta\in\sA_{L,p}}\ \min_{u\in H^1_0(\Om)\setminus\{0\}}E(\theta,u).$$
Interchanging the max and the min above gives the inequality
\be\label{interch}
\max_{\theta\in\sA_{L,p}}\ \min_{u\in H^1_0(\Om)\setminus\{0\}}E(\theta,u)\le\min_{u\in H^1_0(\Om)\setminus\{0\}}\ \max_{\theta\in\sA_{L,p}}E(\theta,u).
\ee
The maximum with respect to $\theta\in\sA_{L,p}$ at the right-hand side above is easily computed and for every fixed $u\in H^1_0(\Om)\setminus\{0\}$ this maximum is reached at
$$\theta=L\frac{|\nabla u|^{2/(p-1)}}{\||\nabla u|^{2/(p-1)}\|_{L^p}}\;.$$
Then the right-hnd side in \eqref{interch} becomes the auxiliary minimization problem
$$\min_{u\in H^1_0(\Om)\setminus\{0\}}\cE(u)\;.$$
A straightforward application of the direct methods of the calculus of variations gives the existence of an optimal solution $u_p$ of the auxiliary problem above.
Setting
$$\theta_p=L\frac{|\nabla u_p|^{2/(p-1)}}{\||\nabla u_p|^{2/(p-1)}\|_{L^p}}$$
by \eqref{interch} we obtain
$$\min_{u\in H^1_0(\Om)\setminus\{0\}}E(\theta_p,u)\le\max_{\theta\in\sA_{L,p}}\ \min_{u\in H^1_0(\Om)\setminus\{0\}}E(\theta,u)\le\min_{u\in H^1_0(\Om)\setminus\{0\}}\cE(u)=\cE(u_p)\;.$$
The minimum problem at the left-hand side above has $u_p$ as a solution, as it can be easily verified by performing the corresponding Euler-Lagrange equation. In addition, we have
$$E(\theta_p,u_p)=\cE(u_p),$$
so that finally we obtain the equality
$$\max_{\theta\in\sA_{L,p}}\lambda(\theta)=\lambda(\theta_p)$$
which proves the first assertion. In addition, the function $u_p$ verifies the PDE
\be\label{pdeup}
-\divs\big((1+m\theta_p)\nabla u_p\big)=\lambda(\theta_p)u_p,\qquad u_p\in H^1_0(\Om).
\ee
The fact that $u_p \in L^\infty(\Om)$ is standard (see Remark \ref{remark.fond1}). To prove the H\"{o}lder-regularity, notice that, if $\partial\Om$ is of class $C^{2,\alpha}$, then by \cite{LIEBERMAN19881203} there exists $\beta=\beta(\alpha) \in(0,1)$ such that
$$u_p\in C^{1,\beta}(\bar{\Om}).$$
Next, we can apply Theorem 1.2.12 of \cite{henrot1} to conclude that $u_p\in C^{2,\beta}(\bar{\Om})$. The reason is that the coefficient $1+m\theta_p$ of the PDE in \eqref{pdeup} is H\"older-continuous with a parameter $\beta$ as a consequence of the definition of $\theta_p$ itself.
\end{proof}

The next result shows that we can estimate $\lambda_1(\theta_p)$ uniformly with respect to $p$.

\begin{lemma}\label{lemma.boundeig}
Let $p\ge1$. Then there exists a positive constant $c$ depending only on $\Om$, its volume and $L$ such that
$$\lambda_1(\theta_p)\le c\qquad\text{for all }p\ge1.$$
\end{lemma}

\begin{proof}
Since $\Om$ is bounded, for $u\in W^{1,2q}(\Om)\cap W^{1,\infty}(\Om)$ we have
$$\||\nabla u|^2\|_{L^q(\Om)}\le|\Om|^{1/q}\||\nabla u|^2\|_{L^\infty(\Om)}\le\max\{1,|\Om|\}\|\nabla u\|_{L^\infty(\Om)}^2.$$
Therefore, we can bound the eigenvalue as follows
\[\begin{split}
\lambda_1(\theta_p)
&=\min_{u\in H_0^1(\Om)\setminus\{0\}}\frac{\int_\Om|\nabla u|^2\,dx+mL\||\nabla u|^2\|_{L^q(\Om)}}{\int_\Om|u|^2\,dx}\\
&\le\min_{u\in H_0^1(\Om)\setminus\{0\}}\frac{\int_\Om|\nabla u|^2\,dx+mL\max\{1,|\Om|\}\|\nabla u\|_{L^\infty(\Om)}^2}{\int_\Om|u|^2\,dx}\;.
\end{split}\]
The latter does not depend on $p$, and a straightforward application of the direct methods of the calculus of variations shows that the minimum is achieved and it is finite. It follows that
$$\lambda_1(\theta_p)\le c(\Om,L)$$
for all $p \geq 1$, and this concludes the proof.
\end{proof}

Before passing to prove higher summability and regularity properties of the solutions of the optimization problem \eqref{optpblambda} we need to prove some uniform (with respect to $p$) estimates of the solutions $u_p$ and $\theta_p$. An important step is a $\Gamma$-convergence of the related functionals.

\subsection{$\Gamma$-convergence as $p \to 1$ of the functions $u_p$}

Let $p>1$ and $q=p'$ as before. Consider the family of functionals defined on $L^2(\Om)$
$$F_p(u):=\begin{cases}
\displaystyle\frac{\int_\Om|\nabla u|^2\,dx+mL\||\nabla u|^2\|_{L^q(\Om)}}{\int_\Om|u|^2\,dx}&\text{if }u\in W_0^{1,2q}(\Om),\\
+\infty&\text{otherwise,}
\end{cases}$$
and for $p=1$
$$F_1(u):=\begin{cases}
\displaystyle\frac{\int_\Om|\nabla u|^2\,dx+mL\|\nabla u\|^2_{L^\infty(\Om)}}{\int_\Om|u|^2\,dx}&\text{if }u\in W_0^{1,\infty}(\Om),\\
+\infty&\text{otherwise.}
\end{cases}$$
For $p\ge1$ denote by $u_p$ the unique positive minimizer and with unitary $L^2$ norm of the corresponding functional $F_p$.

\begin{definition}[$\Gamma$-convergence]
Let $(X,d)$ be a metric space. We say that a sequence of functionals
$$F_\eps:X\longrightarrow\R\cup\{+\infty\}$$
$\Gamma$-converges to a functional $F:X\to\R\cup\{+\infty\}$ if the following hold:\begin{enumerate}
\item[(i)]for every sequence $(x_n)$ in $X$ converging to some $x\in X$ we have (often called $\Gamma-\liminf$ inequality)
$$F(x)\le\liminf_{n\to\infty}F_n(x_n);$$
\item[(ii)]for every $x\in X$ there is a sequence $(x_n)$ in $X$ converging to $x$ such that (often called $\Gamma-\limsup$ inequality)
$$F(x)\ge\limsup_{n\to\infty}F_n(x_n).$$
\end{enumerate}
\end{definition}

\begin{proposition}\label{proposition.3.6}
Let $\Om\subset\R^2$ be a bounded set with $|\Om|<+\infty$. Then
\begin{enumerate}[label={(\alph*)}]
\item the sequence of functionals $F_p$ $\Gamma$-converges to $F_1$ in $L^2(\Om)$;
\item the sequence of minima $u_p$ converges strongly in $H^1(\Om)$ to $u_1$ and
$$\lim_{p\to 1^+}\|\nabla u_p\|_{L^{2q}(\Om)}=\|\nabla u_1\|_{L^\infty(\Om)}.$$
\end{enumerate}
\end{proposition}

\begin{proof} 
The same proof given in \cite[Proposition 3.3]{buouve1} works if we set $f:=\lambda(\theta_p)u_p$ because we can estimate it in $H^1(\Om)$ uniformly with respect to $p$. In particular, we have
$$\|\lambda(\theta_p)u_p\|_{L^2(\Om)}=\lambda(\theta_p)\le c$$
as a consequence of Lemma \ref{lemma.boundeig} and
$$\int_\Om|\nabla u_p|^2\,dx=\|\lambda(\theta_p)u_p\|_{L^2(\Om)}-mL\||\nabla u_p|^2\|_{L^q(\Om)}\implies\|\nabla u_p\|_{L^2(\Om)}\le c.$$
in a similar way.
\end{proof}

\subsection{Uniform estimate of $u_p$}

To find a uniform estimate of $u_p$ and $\theta_p$ we need to assume that a certain geometric condition is satisfied by $\Om$, namely that there is a uniform $\rho > 0$ such that the external ball condition (see Definition \ref{def:ebc}) holds with $\rho$ at all $x \in \Om$. Following closely the method developed in \cite{buouve1}, we obtain an almost uniform estimate on the $L^\infty$-norm of the gradient $\nabla u_p$ on the boundary of $\Om$ since the $L^\infty$-norm of $u_p$ must be taken into account. We first recall a few regularity properties.

\begin{lemma}
Let $\Om\subset\R^2$ be a bounded Lipschitz domain and let $p\ge1$. Then
$$u_p\in W^{2,r}(\Om)\qquad\text{for all }r>d.$$
In particular, if $\Om$ is of class $C^{1,1}$ then $u_p\in C^1(\bar{\Om})$.
\end{lemma}

\begin{proof}It follows from \cite[Theorem 9.15]{gitru1} and a standard bootstrap argument.
\end{proof}

\begin{remark}\label{remark.fond1}
The Sobolev embedding theorem \cite{ev} gives us another proof of the fact that $u_p\in L^\infty(\Om)$ and, more precisely, that when $\Om$ is of class $C^{2,\alpha}$,
$$u_p\in C^{1,\beta}(\bar{\Om})\qquad\text{for some }0<\beta=\beta(\alpha)<1.$$
\end{remark}

The main ingredients behind the uniform estimate of $\nabla u_p$ are two weak comparison principles for eigenfunctions and the fact that on a radially symmetric domain the optimal pair $(\theta_p,u_p)$ is radial (see Lemma \ref{lemma:radiale}).

\begin{lemma}[Weak comparison principle]\label{lemma:maxpri}
Let $\Om\subset\R^d$ be a bounded connected open set and let $G:[0,+\infty)\to[0,+\infty)$ be a convex function such that $G'(0)>0$.
\begin{enumerate}[label=\text{(\roman*)}]
\item Denote by $u_\Om$ the unique positive solution of
\[\begin{cases}
-\divs\big(G'(|\nabla u|^2)\nabla u(x)\big)=\lambda_1(\Om,G)u(x)&\text{if }x\in\Om,\\
u(x)=0&\text{if }x\in\partial\Om,
\end{cases}\]
with unitary $L^2$ norm. Then, for any $\omega\subset\Om$ bounded open subset, it turns out that $u_\Om\ge u_\omega$.
\item Let $u_\Om$ be as above. If $\bar{u}$ is the unique positive solution of
\[\begin{cases}
-\divs\big(G'(|\nabla u|^2)\nabla u(x)\big)=\lambda_1(\Om,G)\|u_\Om\|_{L^\infty(\Om)}&\text{if }x\in\Om,\\
u(x)=0&\text{if }x\in\partial\Om,
\end{cases}\]
with unitary $L^2$ norm, then $\bar{u}\ge u_\Om$.
\end{enumerate}
\end{lemma}

\begin{remark}
The assumption $\Omega$ connected ensures that we can choose $u_\Om$ and $u_\omega$ to be the unique positive solutions with fixed $L^2$-norm.
\end{remark}

\begin{lemma}\label{lemma.a.1}
Let $\Om\subset\R^2$ be a bounded open set with boundary in $C_{\mathrm{loc}}^{1,\alpha}$ for some $\alpha>0$. Suppose that $\Om$ satisfies the external ball condition at some $x_0\in\partial\Om$ with radius $\rho>0$. Then there is a constant $c=c(|\Om|,L)$ such that
\be\label{estimate.a.1}
\frac{\left(1+C_p|\nabla u_p|^{2(q-1)}(x_0)\right)}{\|u_p\|_{L^\infty(\Om)}}\le c\left(1+\frac{\diam(\Om)}{\rho}\right)^{d-1}\diam(\Om),
\ee
where
$$C_p:=mL\left(\int_\Om|\nabla u_p|^{2q}(x)\,dx\right)^{-1/p}.$$
\end{lemma}

\begin{proof}Introduce the auxiliary function
$$G(t):=t+\frac{C_p}{q}t^q$$
and notice that it satisfies the assumptions of Lemma \ref{lemma:maxpri}. We can assume without loss of generality that the center of the external ball at $x_0$ is the origin (i.e., $y_0=0$) so that, setting $R:=\diam(\Om)$, we obtain the inclusion
$$\Om\subset B_{R+\rho}\setminus\bar{B}_\rho=:C_{R,\rho}.$$
Let $\tilde{u}$ be the solution of
$$-\divs\big(G'(|\nabla u|^2)\nabla u\big)=\lambda_1(\theta_p)\|u_p\|_{L^\infty(\Om)}$$
with $u\in W_0^{1,2q}(\Om)$ and $L^2$-norm equal to one, and let $U$ be the solution of
$$-\divs\big(G'(|\nabla u|^2)\nabla u\big)=\lambda_1(C_{R,\rho},\theta_p)\|u_p\|_{L^\infty(\Om)}$$
with $u\in W_0^{1,2q}(C_{R,\rho})$ and the same $L^2$-norm as above. Using $(i)$ and $(ii)$ of Lemma \ref{lemma:maxpri} we find that
$$u_p\le\bar{u}\le U$$
since $\Om$ is contained in the annulus $C_{R,\rho}$. The function $U$ is radially symmetric, and therefore we can rewrite the equation in polar coordinates:
\[\begin{cases}
-r^{1-d}\partial_r(r^{d-1}G(|U'|^2)U')=\lambda_1(C_{R,\rho},\theta_p)\|u_p\|_{L^\infty(\Om)}&\text{if }r\in(\rho,\rho+R),\\
U(\rho)=U(\rho+R)=0.
\end{cases}\]
Let $\rho_1$ be the point where $U$ attains its maximum value and integrate the previous equation to deduce that
\[\begin{split}
\rho^{d-1}(1+C_p|U'|^{2(q-1)})U'(\rho)
&=\lambda_1(C_{R,\rho},\theta_p)\int_\rho^{\rho_1}r^{d-1}\|u_p\|_{L^\infty(\Om)}\,dr\\
&\le\lambda_1(C_{R,\rho},\theta_p)\|u_p\|_{L^\infty(\Om)}\frac{R(R+\rho)^d}{d},
\end{split}\]
and next we estimate the left-hand side using that the radial derivative $\partial_r U$ is positive. Finally, the inclusion property of eigenvalues,
$$\Om\subset C_{R,\rho}\implies\lambda_1(C_{R,\rho},\theta_p)\le\lambda_1(\theta_p),$$
together with Lemma \ref{lemma.boundeig}, allows us to infer that \eqref{estimate.a.1} holds.
\end{proof}

If now $\Omega$ satisfies the uniform external ball condition, we can extend the estimate \eqref{estimate.a.1} to hold for $\|\nabla u_p\|_{L^\infty(\partial\Om)}$ provided that we replace $\rho$ with the largest possible radius.

\begin{lemma}\label{uniform.estimate.gradient}
Let $\Om\subset\R^2$ be a bounded open set with boundary in $C_{\mathrm{loc}}^{1,\alpha}$ for some $\alpha>0$. The following assertions hold:
\begin{enumerate}[label=\text{(\roman*)}]
\item If $\Om$ is convex, then
$$\frac{\left(1+C_p\|\nabla u_p\|_{L^\infty(\partial\Om)}^{2(q-1)}\right)}{\|u_p\|_{L^\infty(\Om)}}\le c\diam(\Om),$$
where $c$ is the constant given in Lemma \ref{lemma.boundeig}.
\item If $\Om$ satisfies the uniform external boundary condition with radius $\rho$, then
$$\frac{\left(1+C_p\|\nabla u_p\|_{L^\infty(\partial\Om)}^{2(q-1)}\right)}{\|u_p\|_{L^\infty(\Om)}}\le c\left(1+\frac{\diam(\Om)}{\rho}\right)^{d-1}\diam(\Om).$$
\end{enumerate}
\end{lemma}

The estimate obtained when $\Om$ is convex is more precise, but it is not enough to infer that the optimal density belongs to $L^\infty(\Om)$ through the method presented in this section. In any case, the $L^\infty$-norm of $u_p$ appears at the denominator: we will see later that this does not lead to any additional problem.

\subsection{Uniform estimate of $\theta_p$}

The next step is to find a uniform estimate for the $L^r$-norm of $\theta_p$; more precisely, we prove that
$$\|\theta_p\|_{L^r(\Om)}\le C,$$
where $C$ is a positive constant that depends on $r$ but not on $p$. It is worth remarking that eigenfunctions are bounded in $L^\infty$, so it is convenient to work with their $L^\infty$ norm only.

\begin{remark}\label{remark.fond}
Let $r\ge1$. Then
$$\|u_p\|_{L^r(\Om)}\le\|u_p\|_{L^\infty(\Om)}|\Om|^{1/r}\le\max\{|\Om|,1\}\|u_p\|_{L^\infty(\Om)},$$
which means that we can always find a uniform estimate of the $L^r$-norm using the $L^\infty$-norm. This will be particularly important in the uniform estimate of $\theta_p$.
\end{remark}

We are now ready to prove the a priori estimate for $\theta_p$. The key lemma below is due to De Pascale-Evans-Pratelli in \cite{paevpr} and the proof is adapted to the eigenvalue case in the spirit of \cite{buouve1}.

\begin{lemma}[De Pascale-Evans-Pratelli]\label{lemma.dep}
Let $\Om\subset\R^d$ be a bounded open set with smooth boundary and let $r\ge2$ be arbitrary. Let $G:[0,+\infty)\to [0,\infty)$ be a convex function with $G'(0)>0$ and let
$$u\in C^1(\bar{\Om})\cap H_{\mathrm{loc}}^2(\Om)$$
be the unique positive solution of
\be\label{eq.3.5}
-\divs(G'(|\nabla u|^2)\nabla u)=\lambda_1 u,
\ee
where $\lambda_1$ depends on $\Om$ and $G$ only, satisfying $u=0$ on $\partial\Om$ and with unitary $L^2$-norm. Then for every $\eps\in(0,1)$ we have the following estimate:
\[\begin{split}
\int_\Om|G'(|\nabla u|^2)|^r|\nabla u|^2\,dx\le
&3\eps\|G'(|\nabla u|^2)\|_{L^r(\Om)}^r+|\Om|\left(\frac{(r-1)^r}{\eps^{r-1}}+\frac{1}{\eps^{2r-1}}\right)\lambda_1^r\|u\|_{L^\infty(\Om)}^{2r}\\
&-\frac{(r-1)^2\|u\|_{L^\infty(\Om)}^2}{\eps}\int_{\partial\Om}HG'(|\nabla u|^2)^r|\nabla u|^2\,d\cH^{d-1},
\end{split}\]
where $H$ is the mean curvature of $\partial \Om$ with respect to the outer normal and $|\Om|$ is the volume of the set $\Om$.
\end{lemma}

\begin{proof}
The assumption $\Om$ smooth implies that the solution $u$ to \eqref{eq.3.5} is smooth up to the boundary, that is, $u\in C^\infty(\bar{\Om})$. Set
$$\sigma:=G'(|\nabla u|^2),$$
and use $\sigma^{r-1}u\in H_0^1(\Om)$ as a test function for the equation \eqref{eq.3.5}. Using the integration by parts formula on the left-hand side, we find the identity
\be\label{eq.3.6}
\int_\Om\sigma^r|\nabla u|^2\,dx+(r-1)\int_\Om u\sigma^{r-1}\nabla u\cdot\nabla\sigma\,dx=\lambda_1 \int_\Om \sigma^{r-1}|u|^2\,dx.
\ee
The right-hand side of the equality can be easily estimated via the H\"older inequality and Remark \ref{remark.fond} as:
$$\lambda_1\int_\Om\sigma^{r-1}|u|^2\,dx\le\lambda_1|\Om|^{1/r}\|u\|_{L^\infty(\Om)}^2\|\sigma\|_{L^r(\Om)}^{r-1}.$$
To estimate the integral $\int_\Om u\sigma^{r-1}\nabla u\cdot\nabla\sigma\,dx$ we can use $\varphi := \divs (\sigma^{r-1}\nabla u)$ as test function for the equation \eqref{eq.3.5}; it turns out that
\be\label{eq.3.7}
\begin{split}
\int_\Om\divs(\sigma\nabla u)&\divs(\sigma^{r-1}\nabla u)\,dx
=-\lambda_1\int_\Om\divs(\sigma^{r-2}\sigma\nabla u)u\,dx\\
&=-\lambda_1\int_\Om\Big(\sigma^{r-2}\divs(\sigma\nabla u)u+(r-2)\sigma^{r-2}(\nabla u\cdot\nabla\sigma)u\Big)\,dx\\
&\le\lambda_1\int_\Om\Big(\lambda_1\sigma^{r-2}|u|^2+(r-2)\sigma^{r-2}|\nabla u\cdot\nabla\sigma||u|\Big)\,dx.
\end{split}\ee
Integrating by parts twice the left-hand side leads to the following chain of equalities,
\[\begin{split}
\int_\Om\divs
&(\sigma \nabla u)\divs(\sigma^{r-1}\nabla u)\,dx\\
&=-\int_\Om\sigma \nabla u \cdot \nabla \left[ \divs (\sigma^{r-1} \nabla u)\right]\,dx + \int_{\partial \Om}\sigma \divs (\sigma^{r-1}\nabla u)u_\nu\,d\cH^{d-1}\\
&=\int_\Om (\sigma u_i)_j (\sigma^{r-1} u_j)_i\,dx+ \int_{\partial \Om} \sigma^r \left( u_\nu \Delta u-u_i u_{ij} \nu_j \right) \,d\cH^{d-1}\\
&=\int_\Om (\sigma u_i)_j (\sigma^{r-1} u_j)_i\,dx+ \int_{\partial \Om} \sigma^r \left( \Delta u-u_{\nu\nu} \right) u_\nu \,d\cH^{d-1}\\
&=\int_\Om (\sigma u_i)_j (\sigma^{r-1} u_j)_i \,dx+\int_{\partial \Om} \sigma^r u_\nu( \Delta u - u_{\nu \nu})\,d\cH^{d-1},
\end{split}\]
where $u_\nu=\nabla u\cdot\nu$ and $u_{\nu \nu} = \mathrm{Hess}(u) \nu \cdot \nu$ are, respectively, the first-order and second-order derivatives in the direction of the exterior normal $\nu$ to $\partial\Om$ and we introduce the short notation $u_j=\partial_j u$. It follows that
\[\begin{split}
\int_\Om \divs (\sigma \nabla u)& \divs (\sigma^{r-1}\nabla u) \,dx
=\int_\Om\sigma^r\|\mathrm{Hess}(u)\|_2^2\,dx+(r-1)\int_\Om\sigma^{r-2}|\nabla u\cdot\nabla\sigma|^2\,dx\\
&+r\int_\Om\sigma^{r-1}\sigma_j u_i u_{i j}\,dx+\int_{\partial\Om}\sigma^r u_\nu(\Delta u-u_{\nu\nu})\,d\cH^{d-1}
\end{split}\]
where the $2$-norm associated to the Hessian matrix is given by
$$\|\mathrm{Hess}(u)\|_2^2:=\sum_{i,j=1}^d u_{ij}^2.$$
Since $u$ is smooth up to the boundary of $\Om$, it is easy to verify that the Laplace operator can be decomposed as
$$\Delta u = u_{\nu\nu} + H u_\nu(x) \quad \text{for all $x \in \partial \Om$},$$
where $H$ denotes the mean curvature of $\partial\Om$. This immediately implies the estimate
\be \label{eq.3.10.a}
\begin{split} (r-1)\int_\Om\sigma^{r-2}|\nabla u\cdot\nabla\sigma|^2\,dx+ & \int_{\partial\Om}H\sigma^r|\nabla u|^2\,d\cH^{d-1}
\\ & \le\int_\Om \divs (\sigma \nabla u) \divs (\sigma^{r-1}\nabla u) \,dx.
\end{split}
\ee
In fact, it is easy to verify that
$$r\int_\Om\sigma^{r-1}\sigma_j u_i u_{ij}\,dx+\int_\Om\sigma^r\|\mathrm{Hess}(u)\|_2^2 \,dx\ge0$$
since the $2$-norm is positive by definition and $\sigma_j u_i u_{ij} \geq 0$ follows from the convexity assumption on $G$. We now plug \eqref{eq.3.10.a} into \eqref{eq.3.7} and use the H\"{o}lder inequality with $p=r/2$ and $q=r/(r-2)$ to obtain the following estimate:
\[\begin{split}
\int_\Om\sigma^{r-2}|\nabla u \cdot \nabla \sigma|^2\,dx
&\le\lambda_1^2\int_\Om|u|^2\sigma^{r-2}\,dx-\int_{\partial\Om}H\sigma^r|\nabla u|^2\,d\cH^{d-1}\\
&\le \lambda_1^2 \|u\|_{L^r(\Om)}^2\|\sigma\|_{L^r(\Om)}^{r-2}-\int_{\partial \Om}H\sigma^r|\nabla u|^2\,d\cH^{d-1}.
\end{split}\]
The conclusion follows by combining the inequalities discovered so far with the identity \eqref{eq.3.6} and applying repeatedly the Young inequality
$$A^\alpha B^\beta\le\eps\alpha A+\eps^{\alpha/\beta}\beta B,$$
which is valid for $\alpha + \beta = 1$. In particular, it turns out that
\[\begin{split}
\int_\Om\sigma^r|\nabla u|^2\,dx
&\le\|u\|_{L^\infty(\Om)}\int_\Om\sigma^{r-1}|\nabla u\cdot\nabla\sigma|\,dx+\lambda_1\|u\|_{L^\infty(\Om)}^2|\Om|^{1/r}\|\sigma\|_{L^r(\Om)}^{r-1}\\
&\le\eps\|\sigma\|_{L^r(\Om)}^r+\frac{(r-1)^2\|u\|_{L^\infty(\Om)}^2}{\eps}\int_\Om\sigma^{r-2}|\nabla u\cdot\nabla\sigma|^2\,dx\\
&\quad+\lambda_1\|u\|_{L^\infty(\Om)}^2 |\Om|^{\frac{1}{r}} \|\sigma\|_{L^r(\Om)}^{r-1}\\
&\le\eps\|\sigma\|_{L^r(\Om)}^r + \lambda_1^2 \frac{(r-1)^2 \|u\|_{L^\infty(\Om)}^2}{\eps} \|u\|_{L^r(\Om)}^2 \|\sigma\|_{L^r(\Om)}^{r-2}\\
&\quad+\lambda_1\|u\|_{L^\infty(\Om)}^2 |\Om|^{\frac{1}{r}} \|\sigma\|_{L^r(\Om)}^{r-1}-\frac{(r-1)^2\|u\|_{L^\infty(\Om)}^2}{\eps}\int_{\partial\Om}H\sigma^r|\nabla u|^2 \,d\cH^{d-1}\\
&\le3\eps\|\sigma\|_{L^r(\Om)}^r+|\Om|\left(\frac{(r-1)^r}{\eps^{r-1}}+\frac{1}{\eps^{2r-1}}\right)\lambda_1^r \|u\|_{L^\infty(\Om)}^{2r}\\
&\quad-\frac{(r-1)^2\|u\|_{L^\infty(\Om)}^2}{\eps}\int_{\partial\Om}H\sigma^r|\nabla u|^2\,d\cH^{d-1},
\end{split}\]
and this concludes the proof of the lemma.
\end{proof}

The De Pascale-Evans-Pratelli lemma holds for domains $\Om$ with smooth boundary, so before passing to the uniform estimate of $\theta_p$ we need to present an approximation argument that allows us to use domains with Lipschitz boundary that only satisfy the uniform external ball condition.

\begin{lemma}\label{lemma.r.3.9}
Let $\Om\subset\R^2$ be a bounded open set satisfying the uniform external ball condition and let $\Om_n$ be a sequence of open sets, with $|\Om_n|<\infty$ and such that
$$\Om\subseteq\Om_n\qquad\text{and}\qquad|\Om_n\setminus\Om|\to0\hbox{ as }n\to\infty.$$
Fix $p\in[1,\infty)$ and let $u_p$ be the minimizer of $F_p$ on $\Om$ and $u_p^n$ the minimizers of $F_p$ on $\Om_n$, all positive with fixed $L^2$-norm equal to one and extended by zero outside $\Om$ and $\Om_n$ respectively. Then
$$u_p^n\to u_p$$
strongly in both $H^1(\R^2)$ and $W^{1,2q}(\R^2)$.
\end{lemma}

This is proved in \cite[Lemma 3.9]{buouve1} for the energy problem, but the same arguments apply to the eigenvalue case. We finally have all the ingredients we need to obtain a uniform estimate of $\theta_p$:

\begin{proposition}\label{proposition.3.13}
Let $\Om\subset\R^2$ be a set of finite perimeter satisfying the uniform external ball condition with radius $R$. For every $r\ge d$, there are constants
$$\delta(\Om):=\delta\qquad\text{and}\qquad C(r,d,\mathrm{Per}(\Om),\lambda_1,\diam(\Om),R,\|u_1\|_{L^\infty(\Om)})= C,$$
satisfying the inequality
\be\label{uniform.estimate}
\|\theta_p\|_{L^r(\Om)} \le C \qquad\text{for all }p\in(1,1+\delta).
\ee 
\end{proposition}

\begin{proof}We first suppose that $\Om$ is smooth. Set
$$ G_p(t) := t + \frac{C_p}{q} t^q. $$
and notice that $u_p$ is the minimizer of the functional
$$ H_0^1(\Om)\ni u \longmapsto \int_\Om G_p(|\nabla u|^2) \,dx $$
with $L^2$-norm equal to one. It is easy to verify that $\theta_p$ is the derivative, that is,
$$ \theta_p = G'_p(|\nabla u_p|^2)-1, $$
and, since $H_{\partial\Om}\ge-1/R$ by smoothness of $\Om$, it follows from De Pascale-Evans-Pratelli (Lemma \ref{lemma.dep}) that
\[\begin{split}
\int_\Om \sigma_p^r |\nabla u_p|^2 \,dx\le
&3\eps\|\sigma_p\|_{L^r(\Om)}^r + |\Om| \left( \frac{(r-1)^r}{\eps^{r-1}}+\eps^{1 - 2r} \right)\lambda_1(G_p)^r \|u_p\|_{L^\infty(\Om)}^{2r}\\
&-\frac{(r-1)^2\|u_p\|_{L^\infty(\Om)}^2}{\eps R} \int_{\partial_\Om} \sigma_p^r |\nabla u_p|^2\,d\cH^{d-1}.
\end{split}\]
It follows from the definition of $G_p$ that $\lambda_1$ is equal to $\lambda_1(\theta_p)$, which is uniformly bounded (Lemma \ref{lemma.boundeig}) by a positive constant $\Lambda$. We Denote by $S$ and $B$ the sets
\[\begin{split}
&S=\big\{x\in\bar{\Om}\ :\ |\nabla u_p(x)|\le\|\nabla u_p\|_{L^{2q}(\Om)}\big\},\\
&B=\big\{x\in\bar{\Om}\ :\ |\nabla u_p(x)|>\|\nabla u_p\|_{L^{2q}(\Om)}\big\},
\end{split}\]
and notice that
$$\sigma_p:=G_p(|\nabla u_p|^2)\le1 + mL\qquad\hbox{on }S.$$
We now estimate separately the three terms on the right-hand side of the inequality above. The first one gives
\[\begin{split}
\|\sigma_p \|_{L^r(\Om)}^r
&=\|\sigma_p \|_{L^r(S)}^r+\|\sigma_p\|_{L^r(B)}^r\\
&\le(1+mL)^r|\Om|+\|\nabla u_p\|_{L^{2q}(\Om)}^{-2}\int_B\sigma_p^r|\nabla u_p|^2\,dx.
\end{split}\]
For the third term we can use the estimate $(ii)$ given in Lemma \ref{uniform.estimate.gradient} to obtain
\[\begin{split}
\int_{\partial \Om}\sigma_p^r|\nabla u_p|^2&\,d\cH^{d-1}=\int_{S \cap \partial \Om} \sigma_p^r|\nabla u_p|^2\,d\cH^{d-1} + \int_{B \cap \partial \Om} \sigma_p^r|\nabla u_p|^2\,d\cH^{d-1}\\
\le&(1+mL)^{r-2} \int_{S \cap \partial \Om} \sigma_p^2 |\nabla u_p|^2 \,d\cH^{d-1}\\
&+\|\nabla u_p\|_{L^{2q}(\Om)}^{2-r}\int_{B\cap\partial\Om}\sigma_p^r|\nabla u_p|^r \,d\cH^{d-1}\\
\le&c^2 (1+mL)^{r-2} \left( 1 + \frac{\diam(\Om)}{\rho} \right)^{2(d-1)} \mathrm{Per}(\Om) \diam^2(\Om) \|u_p\|_{L^\infty(\Om)}^2\\
&+c^r\|\nabla u_p\|_{L^{2q}(\Om)}^{2-r} \left( 1 + \frac{\diam(\Om)}{\rho} \right)^{r(d-1)}\mathrm{Per}(\Om) \diam^r(\Om) \|u_p\|_{L^\infty(\Om)}^r.
\end{split}\]
Now take $\eps =\frac{1}{6} \| \nabla u_p\|_{L^{2q}(\Om)}^2$ in the initial inequality and rearrange the terms in such a way that the following holds:
\[\begin{split}
\frac{1}{2} \int_\Om \sigma_p^r |\nabla u_p|^2\,dx\le
&\frac{1}{2} \|\nabla u_p\|_{L^{2q}(\Om)}^2 (1 + mL)^r |\Om|\\
&+6^{2r-1}|\Om| \left( \frac{(r-1)^r}{\| \nabla u_p \|_{L^{2q}(\Om)}^{2r-2}}+ \frac{1}{\|\nabla u_p\|_{L^{2q}(\Om)}^{4r-2}}\right) \Lambda^r\|u_p\|_{L^\infty(\Om)}^{2r}\\
&+\frac{6(r-1)^2 \|u_p\|_{L^\infty(\Om)}^2}{\|\nabla u_p\|_{L^{2q}(\Om)}^2 R}\Big[ (1 + mL)^{r-2} C_2(\Om) \|u\|_{L^\infty(\Om)}^2\\
&+C_r(\Om) \|\nabla u_p\|_{L^{2q}(\Om)}^{2-r}\|u_p\|_{L^\infty(\Om)}^r\Big],
\end{split} \]
where
$$ C_s(\Om):=c^s \left( 1 + \frac{\diam(\Om)}{R} \right)^{s(d-1)} \mathrm{Per}(\Om) \diam^s(\Om). $$
Now notice that the same inequality holds for any $\Om$ as in the statement since we can approximate it by a sequence of smooth sets $\Om_n$ satisfying the assumptions of Lemma \ref{lemma.r.3.9} and also
$$\mathrm{Per}(\Om_n) \le 2\, \mathrm{Per}(\Omega) \qquad \text{and} \qquad \diam(\Omega_n) \le 2\, \diam (\Omega). $$
Furthermore, since $u_1$ is the limit of $u_p$, we can always find a small positive number $\delta > 0$ such that the following holds:
$$\frac{1}{2} \|u_1\|_{L^\infty(\Om)} \le \|u_p\|_{L^\infty(\Om)} \le \frac{3}{2} \| u_1 \|_{L^\infty(\Om)} \qquad\text{for all }p\in(1,1+\delta_1).$$
Then the estimate above shows that
$$\|\sigma_p\|_{L^r(\Om)}\le C\qquad\text{for all }p\in(1,1+\delta).$$
Finally, from the inequality $\int_\Om\theta_p^r\,dx\le\int_\Om\sigma_p^r\,dx$ and applying once again the estimate
$$ \int_\Om\sigma_p^r\,dx \le (1+mL)^r |\Om| + \| \nabla u_p \|_{L^{2q}(\Om)}^{-2} \int_B \sigma_p^r |\nabla u_p|^2 \,dx.$$
the conclusion follows 
\end{proof}

\begin{remark}\label{remark.fond2} The estimate \eqref{uniform.estimate} is not uniform with respect to $r$ since the constant $C$ depends on $r$ and, while it is true that we can write
$$ \|u_1\|_{L^r(\Om)} \le (1 +|\Omega|) \|u_1\|_{L^\infty(\Om)}, $$
it is also easy to show that
$$ \lim_{r \to \infty} C(r) = \infty.$$
The reason is that there is a term multiplied by a positive constant that is linear with respect to $r$, namely
$$C(r)\simeq\left[ 6^{2r-1}|\Om|\frac{\Lambda^r(r-1)^r}{\|\nabla u_p\|_{L^{2q}(\Om)}^{2r-2}} \|u_p\|_{L^\infty(\Om)}^{2r}\right]^{\frac{1}{r}}\simeq\Lambda r\qquad\text{as }r\to\infty.$$
Thus, even if we assume that $\Om$ is convex, we cannot get rid of this term because it does not come out of the boundary part. In Section \ref{section:aa}, we show an alternative approach to the problem that allows us to achieve $r = \infty$ when $\Om$ is convex, but it requires regularity results in optimal transport theory.
\end{remark}

\subsection{Proof of Theorem \ref{thopt1}: $L^p$-regularity for $1 \leq p < \infty$}
\label{subsec:sproofrec}

We are now ready to use all the results we collected so far to give a proof of the main theorem except for ({\romannumeral 1}) that will be proved in the next section.

\begin{proof}[Proof of Theorem \ref{thopt1}] For $p > 1$, let $u_p \in W_0^{1, 2q}(\Om)$ be the positive minimizer of $F_p$ with fixed $L^2$-norm (equal to one) and let $\theta_p$ be given by
$$ \theta_p(x) = L \left[|\nabla u_p|^{2(q-1)}(x) \left(\int_\Om |\nabla u_p|^{2q} \,dx\right)^{- \frac{1}{p}} \right]. $$
Since $\theta_p$ is admissible, using Proposition \ref{proposition.3.13} we can find a constant $C > 0$ such that for $p > 1$ small enough we have
$$m\|\theta_p\|_{L^p(\Om)} = mL \qquad \text{and} \qquad \| \theta_p \|_{L^r(\Om)} \le C $$
for some $r \geq d$. It follows that $\theta_p$ is uniformly bounded in $L^2(\Om)$ and hence, up to subsequences, it converges weakly to a nonnegative function $\bar{\theta} \in L^2(\Om)$. Since
$$\int_\Om\bar{\theta}\,dx=\lim_{p\to 1^+}\int_\Om \theta_p\,dx\le\liminf_{p\to1^+}\|\theta_p\|_{L^p(\Om)}|\Om|^{1/q}=L,$$
we easily infer that $\bar{\theta}$ is admissible. On the other hand, Proposition \ref{proposition.3.6} asserts that $u_p$ converges strongly in $H_0^1(\Om)$ to the minimum $u_1$ of $F_1$, which means that
$$ \int (1+m \theta_p) \nabla u_p \cdot \nabla \varphi\,dx\to\int_\Om (1+ m \bar{\theta}) \nabla u_1 \cdot \nabla \varphi\,dx\qquad\hbox{as }p\to1^+ $$
for every $\varphi \in C_c^\infty(\Om)$. Moreover, it is easy to check that
$$ \lambda_1(\theta_p) \int_\Om u_p \varphi\,dx\to\lambda_1(\bar{\theta}) \int_\Om u_1 \varphi\,dx\qquad\hbox{as }p\to1^+,$$
which implies that $u_1$ is a solution of the equation
$$ -\divs((1+\bar{\theta}) \nabla u_1) = \lambda_1(\bar{\theta}) u_1 $$
for $x \in \Om$, with Dirichlet boundary condition on $\partial \Om$. An application of the integration by parts formula shows that
$$ E(\bar{\theta}, u_1) := \int_\Om (1+ \bar{\theta}) |\nabla u_1|^2\,dx= \lambda(\bar{\theta}) \underbrace{\int_\Om |u_1|^2\,dx}_{= 1}. $$
The strong converge of $u_p$ to $u_1$ in $L^2(\Om)$ and Proposition \ref{proposition.3.6} implies that
\[\begin{split}
E(\bar{\theta}, u_1)
&=\lim_{p\to1^+}\lambda(\theta_p)\\
&=\lim_{p\to1^+}E(\theta_p,u_p)=\lim_{p\to1^+}F_p(u_p)=F_1(u_1),
\end{split}\]
which means that
$$ E(\bar{\theta},u_1)=\min_{u \in H_0^1(\Om) \setminus \{0\}} F_1(u). $$
Finally the general min-max inequality shows that
\[ \begin{split} \sup_{\theta \in\sA_L} \lambda_1(\theta) & = \sup_{\theta \in\sA_L} \min_{u \in H_0^1(\Om)} E(\theta, u) \\
& \le \min_{u \in H_0^1(\Om)} \sup_{\theta \in\sA_L} E(\theta, u) = \min_{u \in H_0^1(\Om)} F_1(u) = E(\bar{\theta}, u_1), \end{split} \]
and this is enough to conclude that $\bar{\theta}$ is a solution of the maximization problem \eqref{relpb1} since it is an admissible competitor.

The regularity of the optimal density $\bar{\theta}$ follows from the fact that $\bar{\theta} \in L^r(\Om)$ for $r$ arbitrarily large (since $\Om$ is a bounded-volume set). It is now trivial to show that $\bar \theta$ is equal to zero almost everywhere on the set
$$ \big\{x\in\Om\ :\ |\nabla u_{\bar\theta}|(x)<\|\nabla u_{\bar\theta}\|_{L^\infty(\Om)}\big\}, $$
while \eqref{eq.2.5} follows from the min-max inequality which gives an equality evaluated at the optimal couple. Finally, the assertion ({\romannumeral 2}) follows from the regularity of the minimizer $u_1$ of $F_1$ as in Proposition \ref{proposition:3.2}.
\end{proof}

\subsection{Proof of Theorem \ref{thopt1}: $L^\infty$-regularity for $\Om$ convex}
\label{section:aa}

Let $\Om\subset\R^2$ be a bounded open set which is either convex or with a boundary of class $C^{2,\alpha}$. Instead of relaxing the maximization problem to \eqref{optpblambda}, we can study
\be\label{eq.3.10}
\max_{\mu\in\sA_L}\inf_{u\in C_c^1(\Om)\setminus\{0\}}\left\{\frac{\int_\Om|\nabla u|^2 \,dx+m\int_\Om|\nabla u|^2\,d\mu}{\int_\Om|u|^2\,dx} \right\},
\ee
where $\sA_L$ is the class defined in \eqref{relpb1}. Set
\[
J(u,\,\mu) := \left\{\frac{\int_\Om|\nabla u|^2 \,dx+m\int_\Om|\nabla u|^2\,d\mu}{\int_\Om|u|^2\,dx} \right\},
\]
and consider the functional
\be \label{eq.3.11} J(\mu) := \inf_{u \in C_c^1(\Om) \setminus \{0\}} \left\{\frac{\int_\Om|\nabla u|^2 \,dx+m\int_\Om|\nabla u|^2\,d\mu}{\int_\Om|u|^2\,dx} \right\}. \ee
Notice that there are no measures in $\sA_L$ such that $J(\mu) = - \infty$. This is a significant advantage over the energy problem and can be checked easily since
$$ \frac{\int_\Om|\nabla u|^2\,dx+m\int_\Om|\nabla u|^2\,d\mu}{\int_\Om|u|^2\,dx}\ge \frac{\int_\Om|\nabla u|^2\,dx}{\int_\Om|u|^2\,dx}\ge\lambda_1(\Om)>0.$$

\begin{proposition} \label{prop.4.1} Let $\Om\subset\R^2$ be a bounded open set. Then the maximization problem \eqref{eq.3.10} admits a solution $\bar{\mu} \in\sA_L$ with
\be\label{eq.3.12}
\spt\mu\subset\{x\in\Om\ :\ |\nabla \bar{u}(x)| = \| \nabla \bar{u} \|_{L^\infty}\},
\ee
where $\bar{u}$ is the unique positive function with fixed $L^2$-norm achieving the minimum in the functional \eqref{eq.3.11}. Furthermore, we have the identity
\be\label{eq.3.13}
J(\bar{\mu})=\min_{u\in W_0^{1,\infty}(\bar{\Om})\setminus\{0\}}\frac{\int_\Om|\nabla u|^2\,dx+mL\|\nabla u\|^2_{L^\infty(\Om)}}{\int_\Om |u|^2\,dx}.
\ee
\end{proposition}

\begin{proof}First, we notice that for every fixed $u \in C_c^1(\Om)$ the map $\mu \mapsto J(u,\mu)$ is continuous with respect to the weak* convergence. Hence, being $J(\mu)$ the infimum of continuous maps, it is weakly* upper semi-continuous. Since we observed already that $\sA_L$ is weakly* compact and nonempty we infer that \eqref{eq.3.11} admits a solution $\bar{\mu}$. To prove the other claims, observe that
$$ \max_{\mu\in\sA_L}\inf_{u \in C_c^1(\Om)\setminus\{0\}} J(u,\mu) \le \inf_{u \in C_c^1(\Om) \setminus \{0\}} \sup_{\mu \in\sA_L} J(u,\mu) $$
is always true, but the equality does not hold a priori since we lack concavity of $J$ with respect to the first variable $u$. Nevertheless, we have
$$ \sup_{\mu\in\sA_L}J(u,\mu)=\frac{\int_\Om|\nabla u|^2\,dx+mL\|\nabla u\|^2_{L^\infty(\Om)}}{\int_\Om |u|^2\,dx}, $$
and it is easy to check that this is achieved using any measure $\mu$ with total mass equal to $L$ and support satisfying \eqref{eq.3.12}. Denote any one of them by $\bar{\mu}$ and notice that $\bar{\mu}\in\sA_L$ implies
\[\begin{split}
\max_{\mu\in\sA_L}\inf_{u\in C_c^1(\Om)\setminus\{0\}}J(u,\mu)
&\ge\inf_{u\in C_c^1(\Om)\setminus\{0\}}\frac{\int_\Om|\nabla u|^2\,dx+m\int_\Om |\nabla u|^2\,d\bar{\mu}}{\int_\Om |u|^2 \,dx}\\
&=\inf_{u\in C_c^1(\Om)\setminus\{0\}}\frac{ \int_\Om |\nabla u|^2 \,dx+ mL \|\nabla u\|_{L^\infty(\Om)}^2}{\int_\Om |u|^2 \,dx}\\
&=\inf_{u\in C_c^1(\Om)\setminus\{0\}}\sup_{\mu \in\sA_L} J(u, \, \mu).
\end{split}\]
This shows that we can interchange infimum and supremum and therefore the claim \eqref{eq.3.13} holds, concluding the proof.
\end{proof}

Now that we know the existence of the optimal measure $\bar{\mu}$, we can equivalently investigate the minimization problem associated with the functional
$$J_1(u) = \int_\Om |\nabla u|^2\,dx+ mL \|\nabla u\|_{L^\infty(\Om)}^2$$
satisfying the additional constraint $\int_\Om |u|^2 \,dx=1$. The proof of the next result follows immediately from \cite{ev1}. Notice that what we find out here is compatible with the investigation we carried out in Section \ref{sproofs}.

\begin{theorem}\label{thm.1.1.1}
The optimization problem
$$ \min\big\{J_1(u)\ :\ u\in H_0^1(\Om),\ \|u\|_{L^2(\Om)}=1\big\} $$
admits a unique solution $\bar{u}\in W^{2,p}(\Om)$ for all $p>d$. If, in addition, $\Om$ is convex, then $\bar{u}\in W^{2,\infty}(\Om)$.
\end{theorem}

We can now show that the optimal measure $\bar{\mu}$ belongs to $L^p$ spaces for $p<\infty$ arbitrary, as in Section \ref{subsec:sproofrec}, and show that we can obtain a uniform $L^\infty$ estimate if $\Om$ convex concluding the proof of Theorem \ref{thopt1}.

\begin{proof}[Proof of Theorem \ref{thopt1} ({\romannumeral 1})] Let $\bar{\mu}$ be the optimal measure given by Proposition \ref{prop.4.1}. A standard result in elliptic regularity theory (e.g, \cite{stbr1}) implies
$$ \bar{u} \in \mathrm{argmin}(J_1(\cdot)) \implies \bar{u} \in C^{2, \, \beta}(\Om) $$
for some $\beta > 0$. Note that here we use the fact that $\Om$ is either regular ($C^{2, \, \alpha}$ and thus $\beta$ depends on $\alpha$) or convex. By Theorem \ref{thm.1.1.1}, we infer that
$$\Delta u \in L^p(\Om) \qquad \text{for all $p \geq 1$}.$$
Thus $\bar{u}$ and $\bar{\mu}$ solve the problem
\[\begin{cases}
- \divs( \bar{\mu} \nabla \bar{u} ) = \Delta \bar{u} + \lambda_1\bar{u}, & \text{if $x \in \Om$},
\\[.5em] \bar{u} \, \big|_{\partial \Om} \equiv 0,
\\[.5em] |\nabla \bar{u}(x)| = \|\nabla \bar{u} \|_{L^\infty(\Om)}, & \text{if $x \in\spt(\bar{\mu})$}.
\end{cases}\]
The right-hand side of the equation belongs to $L^p(\Om)$ so we deduce that also $\bar{\mu} \in L^p(\Om)$ for all $p > d$ using the regularity results for the Monge-Kantorovich problem given in \cite{paevpr, papr1, sa1}. Similarly, we have
$$\Om\text{ convex}\implies\Delta\bar{u}\in L^\infty(\Om),$$
and we can apply once more \cite{paevpr, papr1, sa1} to conclude that $\bar{\mu}\in L^\infty(\Om)$.
\end{proof}

\begin{remark}
The approach via $\Gamma$-convergence and the approach presented in this section are both needed to prove Theorem \ref{thopt1}. Indeed, the $L^\infty$ estimate on $\bar{\theta}$ is impossible to obtain via $\Gamma$-convergence, even if we assume that $\Om$ is convex. On the other hand, the Monge-Kantorovich approach requires $\Om$ to be either convex or $C^{2,\alpha}$ even for the higher integrability result $\theta \in L^p(\Om)$, which only requires the uniform external ball condition with $\Gamma$-convergence.
\end{remark}

\section{The radial case}\label{sradial}

Let $\Om$ be the unit disc of $\R^2$. In this section, we exploit the symmetries of the domain to show that the solution of \eqref{relpb1} is a radially symmetric function with an explicit formula. First, we prove a technical lemma which allows us to use polar coordinates to deal with the min-max problem.

\begin{lemma} \label{lemma:radiale}
The optimal density $\bar{\theta}$ solution of \eqref{relpb1} and the corresponding optimal profile $\bar{u} := u_{\bar{\theta}}$ are both radially symmetric functions.
\end{lemma}

\begin{proof}
Let $u_p$ be the function given in Proposition \ref{proposition:3.2}. Then $u_p$ is the unique solution (with fixed $L^2$-norm) of the minimization problem
$$ \min_{u \in H_0^1(\Om)} \left\{ \frac{\int_\Om |\nabla u|^2 \,dx+mL\| \nabla u \|_{L^{2q}(\Om)}^2}{\int_\Om |u|^2 \,dx} \right\}. $$
Now recall that the Steiner symmetrization \cite[Chapter 7]{stha1} of a function $u\in H_0^1(\Om)$, denoted by $u^\ast$, satisfies the P\'olya-Szeg\"o's inequality
\[\int_{\Om^\ast}|\nabla u^\ast|^p\,dx\le\int_\Om |\nabla u|^p\,dx.\]
for all $1\le p<\infty$ (see \cite[Theorem 2.2.4]{henrot1}). The unit ball is symmetric so it coincides with its symmetrization $\Om^\ast$ and the $L^2$-norm of $u$ coincide with the one of $u^\ast$ so from the inequality
$$
\frac{\int_\Om|\nabla u^\ast|^2\,dx+mL\||\nabla u^\ast|^2\|_{L^q(\Om)}}{\int_\Om|u^\ast|^2\,dx} \le \frac{\int_\Om|\nabla u|^2\,dx+mL\||\nabla u|^2\|_{L^q(\Om)}}{\int_\Om|u|^2\,dx}
$$
we infer that each $u_p$ is radially symmetric. On the other hand, we proved in Proposition \ref{proposition.3.6} that $u_p$ converges strongly in $L^2$ to $\bar{u}$ as $p \to 1$ and therefore, up to subsequences, we can assume that $u_p$ converges almost everywhere to $\bar{u}$. Thus
$$
\bar{u}(x) = \lim_{p \to 1^+} u_p(x) =\lim_{p \to 1^+} u_p(|x|) \implies \text{$\bar{u}$ radially symmetric.}
$$
Finally, if we choose $\bar{u}$ to be the solution with $L^2$-norm equal to one, then $\bar{\theta}$ is the unique maximizer of the functional
$$ \theta \longmapsto \int_\Om (1 + \theta)\nabla \bar{u}\,dx.$$
The function $\bar{u}$ is radial so we can apply ({\romannumeral 4}) of \cite[Theorem 2.2.4]{henrot1} and obtain
$$\int_\Om (1 + \theta) \nabla \bar{u}\,dx\le \int_\Om (1+\theta)^\ast \nabla \bar{u} \,dx,$$
which allows us to conclude that $\bar{\theta}$ is also radially symmetric.
\end{proof}

Now fix $d = 2$ for simplicity and notice that the optimal profile $\bar{u}$ and the optimal density $\bar{\theta}$ satisfy the elliptic equation
\be \label{eq.3.8.c}
-\divs((1+\bar{\theta})\nabla \bar{u}) = \lambda_1 \bar{u}
\ee
so, exploiting the fact that they are both radially symmetric, we can write
\be \label{eq.3.8.a}
- \frac{1}{r} \partial_r \big( r (1 + \bar{\theta}(r)) \bar{u}'(r) \big) = \lambda_1 \bar{u}(r).
\ee
Since by Theorem \ref{thopt1} the support of $\bar{\theta}$ is contained in the set where $|\nabla \bar{u}|$ achieves its maximum value, it is easy to see that there exists $\bar{a} \in (0,\, 1)$ such that
$$
\bar{u}(r) = 1-r \quad \text{for all $r \in [\bar{a},\, 1]$}.
$$
Notice that we can avoid placing a constant in front of $(1 - r)$ because we are using that $\bar{u}$ is unique when the $L^2$-norm is fixed (although we do not care about the actual value here). It follows from \eqref{eq.3.8.a} that
$$
\frac{1}{r} \partial_r \big( r (1 + \bar{\theta}(r)) \big) = \lambda_1 (1-r) \quad \text{for all $r \in [\bar{a}, \, 1]$},
$$
and this leads to a ordinary differential equation in $\theta$ which admits an explicit solution that depends on $\bar{a}$ and $\lambda_1$, that is,
\[
\theta(r) = \begin{cases} 0 & \text{if $r \in [0,\, \bar{a}]$},
\\[.6em] - \frac{\lambda_1}{3}r^2 + \frac{\lambda_1}{2}r - 1 + \frac{\bar{a}}{r} \left( 1 + \frac{\lambda_1}{3} \bar{a}^2 - \frac{\lambda_1}{2} \bar{a} \right) & \text{if $r \in [\bar{a},\, 1]$}. \end{cases}
\]
Suppose that the value of $\lambda_1$ is known. We can find the optimal value $\bar{a}$ by exploiting the integral condition on $\theta$. Namely, we know that $\int_\Omega \theta \, dx = L$ so
$$
\int_{\bar{a}}^1 \theta(r) r \, dr = \frac{L}{2 \pi} \implies \lambda_1(\bar{a}) =12 \left( \frac{\frac{L}{2 \pi} + \frac{1}{2}(\bar{a}-1)^2}{1 - 6\bar{a}^2 + 8 \bar{a}^3 - 3 \bar{a}^4} \right).
$$
This leads to
\be \label{eq.3.8.b}
\lambda_1 = \left( \frac{\frac{L}{2 \pi} + \frac{1}{2}(\bar{a}-1)^2}{1 - 6\bar{a}^2 + 8 \bar{a}^3 - 3 \bar{a}^4} \right) \implies \bar{a} = f^{-1}(\lambda_1),
\ee
which admits a unique solution in the interval $(0,\, 1)$ provided that $\lambda_1$ is bigger than or equal to the minimum of $f$, which is true for the one given in \eqref{eq.3.8.c}.

\begin{remark}
In the energy problem with $f = 1$, one can prove (see Example 5.1 of \cite{buouve1}) that $\bar{a}$ is the unique solution of the polynomial equation
$$ 
a^{d+1} - (d+1) \left( 1 + \frac{mL}{\omega_d} \right)a + d = 0. 
$$
Observe that for $L = 0$ the unique solution of the equation is $\bar{a} = 1$, and this is compatible with the fact that there is no reinforcement at all. The same is true in the eigenvalue case, but it cannot be inferred from \eqref{eq.3.8.b} since the relation holds only when a density appears.
\end{remark}

We can now recover $\bar{u}$, the optimal profile, completely using the boundary condition naturally arising from the decomposition and the Neumann condition at the origin. Namely, it is easy to verify that
$$
\bar{u}(r) = c_1 J_0( \sqrt{\lambda_1} r) + c_2 Y_0 (\sqrt{\lambda_1} r) \quad \text{for $r \in [0,\, \bar{a}]$},
$$
where $J_0$ and $Y_0$ are the first Bessel functions of first and second kind respectively. To find the constants we simply notice that by continuity
$$
c_1 J_0( \bar{a} \sqrt{\lambda_1}) + c_2 Y_0 (\bar{a} \sqrt{\lambda_1} ) = 1-\bar{a}
$$
and, similarly, the Neumann condition gives
$$
\lim_{r \to 0^+} \left[ c_1 J_1( \sqrt{\lambda_1} r) + c_2 Y_1 (\sqrt{\lambda_1} r) \right] = 0.
$$
Since $\lim_{r \to 0^+} Y_1(r) = - \infty$ and $\lim_{r \to 0^+} J_1(r) = 0$, the second condition is satisfied if and only if $c_2 = 0$. This immediately shows that
$$
c_1 = \frac{1-\bar{a}}{J_0(\bar{a} \sqrt{\lambda_1} )},
$$
and therefore the optimal profile is given by
$$
\bar{u}(r) = \begin{cases} \frac{1-\bar{a}}{J_0(\bar{a} \sqrt{\lambda_1} )} J_0( \sqrt{\lambda_1} r)& \text{if $r \in [0,\, \bar{a}]$},
\\[.6em] 1 - r & \text{if $r \in [\bar{a},\,1]$}. \end{cases}
$$
We now show the shape of the optimal density $\bar{\theta}$ (see Figure \ref{figure:1}) via a numerical analysis. The general idea is to fix $\lambda_1$ admissible (bigger than $j_{0,0}^2$) and recover the unique $\bar{a} \in (0,\, 1)$ from the minimum problem
$$
\min_{a \in (0,\,1)} \frac{ \left[ \sqrt{\lambda_1} \frac{1-\bar{a}}{J_1(\bar{a} \sqrt{\lambda_1})} \right]^2\int_0^a r \left[J_0( \sqrt{\lambda_1} r) \right]^2 \, dr + \frac{1}{2}(1-a^2) + \frac{mL}{2 \pi}}{\left[ \frac{1-\bar{a}}{J_0(\bar{a} \sqrt{\lambda_1})} \right]^2\int_0^a r \left[J_0( \sqrt{\lambda_1} r) \right]^2 \, dr + \int_a^1 r(1-r)^2 \, dr}.
$$
Finally the length $L$ is determined starting from the identity \eqref{eq.3.8.b}. The numerical simulation confirms the regularity result in Theorem \ref{thopt1} because $\bar{a} \in (0,\,1)$ turns out to be the unique one for which
$$
- \frac{(1-\bar{a})\sqrt{\lambda}}{J_0(\bar{a} \sqrt{\lambda_1} )} J_1( \sqrt{\lambda_1} r) = \lim_{r \to \bar{a}^-} \bar{u}'(r) = \lim_{r \to \bar{a}^+} \bar{u}'(r) = -1
$$
holds. We would like to point out the main difference with the energy problem. In Example 5.1 of \cite{buouve1} it was proved that the optimal density is linear,
$$\bar{\theta}_f(r)=\Big(\frac{r}{\bar a}-1\Big)^+\qquad\forall r\in[0,1],$$
while in our case the optimal density is not linear (it depends on $r^2$ and $r^{-1}$) and, coherently with this dependence, it is not strictly increasing but rather
$$
\exists \, \bar{r} \in (\bar{a},\, 1) \: : \: \text{$\theta \, \big|_{(\bar{a},\, \bar{r})}$ increasing and $\theta \, \big|_{(\bar{r},\, 1)}$ decreasing}.
$$

\begin{figure}[!tb] 
\begin{minipage}{0.48\textwidth}
\centering
\includegraphics[width=.95\linewidth]{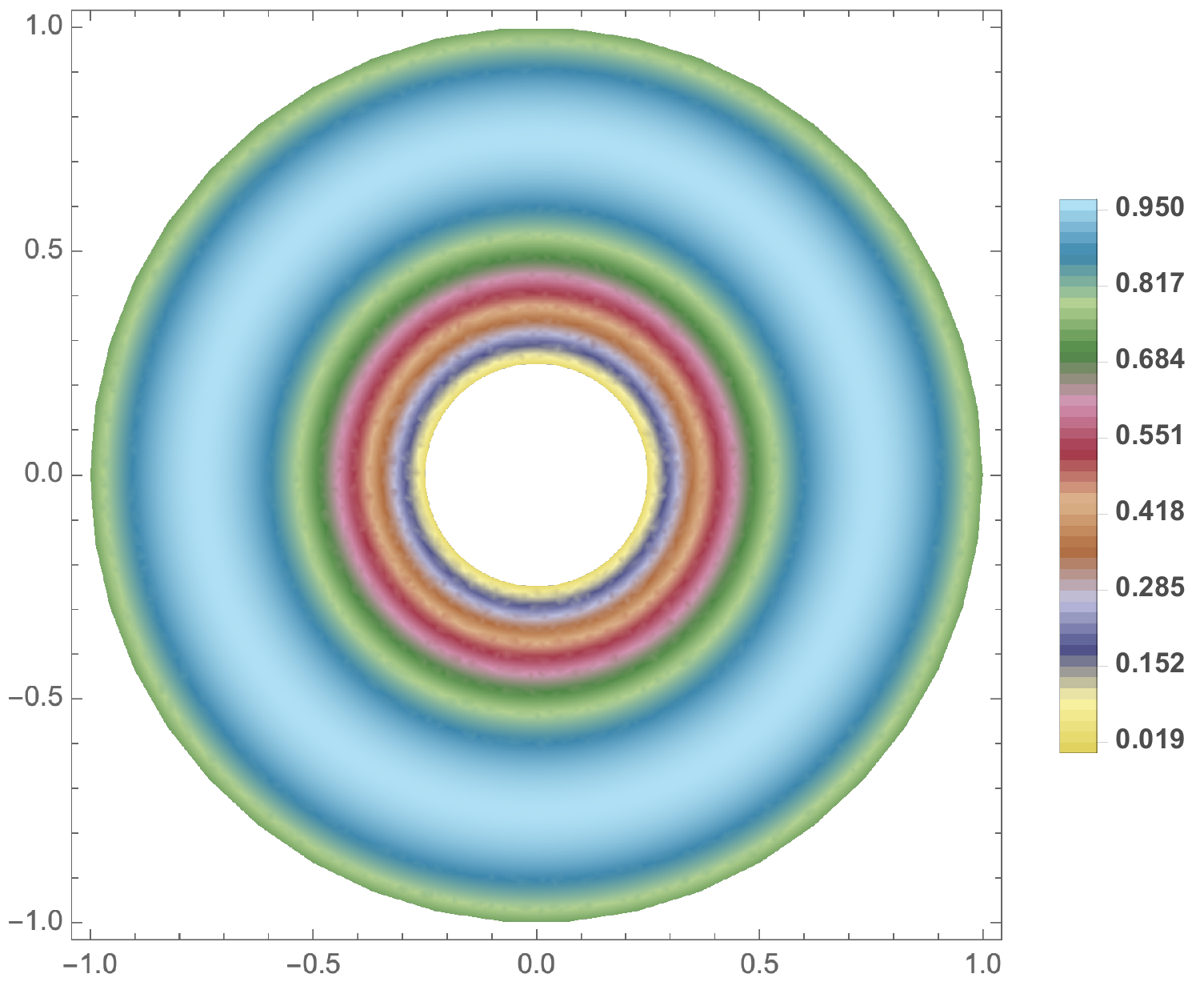}
\end{minipage}\hfill
\begin{minipage}{0.48\textwidth}
\centering
\includegraphics[width=.95\linewidth]{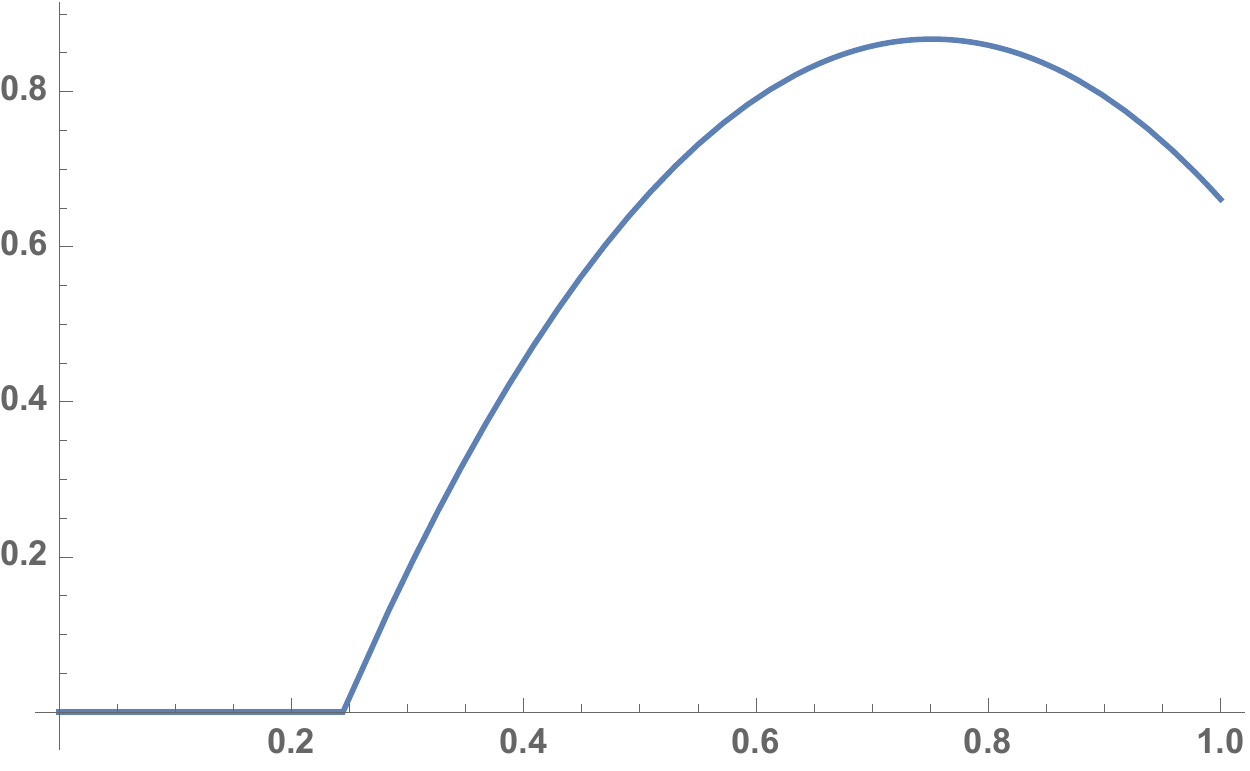}
\end{minipage}
\caption{The picture on the left represents the levels sets of $\theta$ on $B(0,\,1)$. We used $\lambda_1 = 10$ and $m = 5$ and obtained $\bar{a} \approx 0.244419$ and $L \approx 0.424242$. The function $\theta$ is increasing up to $r \approx 0.751491$ and then it starts decreasing up to the border.}
\label{figure:1}
\end{figure}

\section{The connected case}\label{sec:4}

We now consider the maximization problem \eqref{optpb2}, in which $S$ ranges in the class of closed, connected, one-dimensional subsets of $\bar\Om$. We follow closely the method introduced in \cite{abgo} for the same optimal reinforcement problem when an external force acts on the membrane $\Om$, and we show that a small modification of the main proof is enough to get the same conclusion in the eigenvalues' problem.

\subsection{Proof of Theorem \ref{thopt2}}

In Proposition \ref{exrel2} we proved that there exists a solution $\mu$, in the class $\sA_L^c$, to the maximization problem \eqref{relpb2}. It only remains to show that there is some $\theta \in L^1(\Om)$ such that $\mu = \theta \cH^1 \restr S$. The following technical result was proved in \cite[Lemma 3.3]{abgo}.

\begin{lemma}Let $K$ be a compact set in $\R^2$ with $|K| = 0$. For all $\eps > 0$ there exists a function $\phi_\eps$ of class $C^\infty$ satisfying the following properties:
\begin{enumerate}[label=(\arabic*)]
\item $\phi_\eps$ is locally constant on $K$;
\item $|\phi_\eps(x) - x| \le \eps$ at all points $x \in \R^2$;
\item $|\nabla \phi_\eps(x)| \le 1$ for all $x \in \R^2$ and $|\nabla \phi_\eps(x)| = 1$ everywhere except in an open set $A_\eps$ of measure less than $\eps$ containing $K$.
\end{enumerate}
\end{lemma}

We shall now prove that the optimal measure is absolutely continuous with respect to $\cH^1$. In \cite[Lemma 3.4]{abgo} the same result is obtained in the energy problem, so it is sufficient to estimate the denominator and conclude in the same way.

\begin{proof}[Proof of Theorem \ref{thopt2}] The main difference with the mentioned paper is that we will show that for all $v \in C_c^\infty(\Om)$ and $\eps > 0$ there exists $u \in H_0^1(\Om)$ such that
$$\frac{ \int_\Om |\nabla u|^2\,dx+ m \int_\Om |\nabla u|^2\,d\mu}{ \int_\Om |u|^2 \,dx} \le \frac{ \int_\Om |\nabla v|^2 \,dx+m \int_\Om |\nabla v|^2 \,d\mu^a + C_1 \eps}{ \int_\Om |v|^2\,dx- C_2 \eps},$$
where $\mu^a$ is the absolutely continuous part of $\mu$. Indeed, when $\eps > 0$ is small enough we can use Taylor approximation theorem to infer that
$$\frac{ \int_\Om |\nabla u|^2\,dx+ m \int_\Om |\nabla u|^2\,d\mu}{ \int_\Om |u|^2\,dx} \le \frac{ \int_\Om |\nabla v|^2\,dx+m \int_\Om |\nabla v|^2 \, d\mu^a}{ \int_\Om |v|^2\,dx} + C'\eps,$$
where $C'=C_1/C_2$ is a positive constant. Now let
$$u (x) := (1 - \theta_\eps(d(x))) v(x) + \theta_\eps(d(x)) v(\phi_\eps(x)),$$
where $d(x):=d(x,S)$ and $\theta$ is a cut-off function identically one on $B(S,\eps)$ and zero on the complement of $B(S,2 \eps)$. In \cite[Lemma 3.4]{abgo} it was proved that
$$\frac{1}{2} \int_\Om |\nabla u|^2\,dx+ \frac{m}{2} \int_\Om |\nabla u|^2\,d\mu \le \frac{C_1}{2} \eps + \frac{1}{2} \int_\Om |\nabla v|^2\,dx+ \frac{m}{2} \int_\Om |\nabla v|^2\,d\mu^a,$$
so we only have to deal with the denominator. A straightforward computation, assuming that $\|v\|_{L^2(\Om)}^2=1$, shows that
\[\begin{split}
\|u\|_{L^2(\Om)}^2
&=\int_\Om \theta_\eps^2(d(x)) \left[v(x)(v(x) - v(\phi_\eps(x))) + v(\phi_\eps(x))(v(\phi_\eps(x)) - v(x)) \right]\,dx+\\
& + 2 \int_\Om \theta_\eps(d(x))v(x)(v(\phi_\eps(x))-v(x))\,dx + 1,
\end{split}\]
which immediately leads to the following estimate. If we denote $A = \left| \|u\|_{L^2(\Om)}^2 - 1 \right|$, then we find that
\[\begin{split}
A&\le \int_\Om\theta_\eps^2(d(x))\big[|v(x)| \left| v(x) - v(\phi_\eps(x)) \right| + |v(\phi_\eps(x))| \left|v(\phi_\eps(x)) - v(x)\right|\big]\,dx\\
&\qquad+2\int_\Om |\theta_\eps(d(x))| |v(x)| \left| v(\phi_\eps(x))- v(x) \right|\,dx\\
&\le C\eps \int_\Om (|v(x)| + |v(\phi_\eps(x))|)\,dx+ 2 C \eps \int_\Om |v(x)|\,dx\\
&\le 4C\|v\|_{L^1(\Om)}\,\eps.
\end{split}\]
Therefore, to conclude the proof, it suffices to set $C_2 := 4 C \|v\|_{L^1(\Om)}$ and continue as in \cite[Lemma 3.4]{abgo}.
\end{proof}

\subsection{Indirect method and boundary points}

Let $S \in \cA_L^c$ and let $u$ be a solution of the minimization problem
\be \label{eq:5.1}
\lambda_1(S) := \min_{u \in H_0^1(\Om)\setminus \{0\}} \frac{ \int_\Om |\nabla u|^2\,dx+ m \int_S \theta |\nabla_\tau u|^2\,d\cH^1}{ \int_\Om |u|^2\,dx}.
\ee
Then for each $v \in H_0^1(\Om)$ there results
$$\frac{d}{d\eps}\Big|_{\eps = 0} \frac{ \int_\Om |\nabla (u+\eps v)|^2\,dx+ m \int_S \theta |\nabla_\tau (u+\eps v)|^2\,d\cH^{1}}{\int_\Om |u+\eps v|^2\,dx} = 0,$$
which is easily seen to be equivalent to
\[\begin{split}
\int_\Om u^2\,dx & \left[ \int_\Om \nabla u \cdot \nabla v\,dx+m \int_S \theta( \nabla_\tau u \cdot \nabla_\tau v) \,d\cH^{1}\right]\\
&-\int_\Om uv\,dx\left[\int_\Om|\nabla u|^2\,dx+m\int_S\theta|\nabla_\tau u|^2\,d\cH^{1}\right]=0.
\end{split}\]
Since $u$ minimizes the functional in \eqref{eq:5.1} we can substitute it with $\lambda_1(S)$ to obtain the following identity which is valid for all $v \in H_0^1(\Om)$,
$$\int_\Om \nabla u \cdot \nabla v\,dx+ m \int_S \theta (\nabla_\tau u \cdot \nabla_\tau v) \,d\cH^{1} - \lambda_1(S) \int_\Om uv \,dx= 0.$$
The integration by parts formula shows that
$$\int_\Om \nabla u \cdot \nabla v \,dx= - \int_\Om \Delta u v \,dx+ \int_S \left[\frac{\partial u}{\partial \nu} \right] v \,d\cH^{1},$$
where $\left[\frac{\partial u}{\partial \nu} \right] := \partial_+ u + \partial_- u$ does not depend on the choice of an orientation and $\partial_{\pm} u$ are, respectively, the positive and negative derivatives of $u$ on $S$. Finally, since
$$\nabla_\tau u \cdot \nabla_\tau v = (\nabla u \cdot \tau)\tau \cdot (\nabla v \cdot \tau) \tau = (\nabla u \cdot \tau) \cdot (\nabla v \cdot \tau),$$
we can integrate by parts the second term and obtain
$$m\int_S\theta\nabla_\tau u \cdot \nabla_\tau v \,d\cH^{1} = - m \int_S \divs_\tau(\theta \nabla_\tau u) v \,d\cH^{1} + m \left[v \theta \nabla_\tau u\right]_{S^\can},$$
where $-\divs_\tau(- \nabla_\tau)$ is the Laplace-Beltrami operator on $S$, and $S^\can$ is the set of terminal-type and branching-type points of $S$.

\begin{proposition}\label{prp:pdeformulation}If $u \in H_0^1(\Om) \cap C^2(\bar{\Omega})$ is a minimum point of \eqref{eq:5.1}, then it solves the following boundary-value problem:
\[
\begin{cases} - \Delta u = \lambda_1(S) u & \text{if $x \in \Om \setminus S$},
\\ u = 0 & \text{if $x \in \partial \Om$,}
\\ \left[\frac{\partial u}{\partial \nu} \right] - m \,\divs_\tau( \theta \nabla_\tau u) = 0 & \text{if $x \in S$},
\\[.6em] \big[\theta \nabla_\tau u \big]_{S^\can} = 0.\end{cases} 
\] \end{proposition}

For points in $S^\can$ the following three situations are all possible and, in the general case, we expect all of them to appear:
\begin{enumerate}[label=\text{(\roman*)}]
\item \textbf{Dirichlet.} If $x \in S^\can\cap \partial \Om$, then $u(x) = 0$.
\item \textbf{Neumann.} If $x \in S^\can$ is a terminal point of $S$, then $\nabla_\tau u(x) = 0$.
\item \textbf{Kirchhoff.} If $x \in S^\can$ is a branching point of $S$, then
\[
\sum_i \nabla_{\tau_i} u^i(x) = 0,
\]
where $u^i$ is the trace of $u$ over the $i$-th branch of $S$ ending at $x$ and $\tau_i$ the corresponding tangent vector.
\end{enumerate}

As a consequence, using Proposition \ref{prp:pdeformulation} and Theorem \ref{thopt2}, it is easy to verify that \cite[Proposition 4.1]{abgo} can be proved in the same way.

\begin{proposition}
Let $\mu$ be a solution of \eqref{relpb2} and let $u$ be the unique positive solution of the associated minimization problem with $L^2$-norm fixed. Then there exists a positive constant $c$ such that
\[
\begin{cases}
|\nabla_\tau u| = c & \text{for $\cH^1$-a.e. $x\in \{ \theta(x) > 1 \}$},\\
|\nabla_\tau u| \le c & \text{for $\cH^1$-a.e. $x\in \{ \theta(x) = 1 \}$}.
\end{cases}
\]
\end{proposition}

\section{Conclusion and open problems}\label{sopen}

In this section we present and discuss some open questions related to the optimization problems we considered.

\begin{problem}
A first problem is related to the regularity of solutions. We have shown (Theorem \ref{thopt1}) that when $\Om$ is regular enough the optimization problem \eqref{relpb1} admits a solution $\mu$ that is indeed a function $\theta\in C^{1,\beta}(\bar\Om)$ for a suitable $\beta\in(0,1)$. It would be interesting to know if additional regularity properties on $\theta$ hold in general. Similarly, the optimization problem \eqref{relpb2} admits a solution $\mu$ which is of the form $\theta\cH^1\restr S$ for a suitable closed connected set $S\subset\bar\Om$ and a function $\theta\in L^1(S)$, with $\cL(S)\le L$ and $\theta\ge1$ on $S$. Even if we could expect that $S$ and $\theta$ are regular enough, at the moment these regularity results are not available and seem rather difficult. In particular, it would be interesting to prove (or disprove) the regularity of the optimal set $S$ up to a finite number of branching points, where the Kirchhoff rule $\sum_i\nabla_{\tau_i}u=0$ holds.
\end{problem}

\begin{problem}
For the optimal set $S$ of problem \eqref{relpb2} several necessary conditions of optimality merit to be investigated, for instance we list the following ones, that look similar to other problems studied in the fields of optimal transport and of structural mechanics (see \cite{buoust02}, \cite{bust03}, \cite{CLLS17}).
\begin{itemize}
\item[(a)]Does $S$ contain closed loops (i.e. subsets homeomorphic to the circle $S^1$)? This should not be the case, even if a complete proof is missing.
\item[(b)]The branching points of $S$ (if any) do have only three branches or a higher number of branches is possible?
\item[(c)]Does the optimal set $S$ intersect always the boundary $\partial\Om$?
\item[(d)]Is it possible that $\cL(S)=L$ or we always have $\cL(S)<L$ and hence $\theta>1$ somewhere on $S$?
\end{itemize}
\end{problem}

\begin{problem}
As stated in the Introduction, passing from a single connected set $S$ to sets with at most a $N$ connected components (with $N$ a priori fixed) does not introduce essential differences in the statements and in the proofs. However, it would be interesting to establish if, in the case when $N$ connected components are allowed, the optimal set $S$ has actually exactly $N$ components.
\end{problem}

\begin{problem}
Finally, the numerical treatment of the optimization problems we considered, present several difficulties, essentially due to the fact that a very large number of local maxima are possible and global optimization algorithms are usually too slow for this kind of problems. In the case of energy optimization, considered in \cite{abgo} and in \cite{buouve1}, some efficient optimization methods have been implemented, but the eigenvalue optimization considered in the present paper seems to present a higher level of complexity.
\end{problem}

\section*{Acknowledgements}
The work of the first author is part of the project 2017TEXA3H {\it``Gradient flows, Optimal Transport and Metric Measure Structures''} funded by the Italian Ministry of Research and University. The first author is members of the ``Gruppo Nazionale per l'Analisi Matematica, la Probabilit\`a e le loro Applicazioni'' (GNAMPA) of the ``Istituto Nazionale di Alta Matematica'' (INDAM).

\bigskip
{\small\noindent
Giuseppe Buttazzo:
Dipartimento di Matematica, Universit\`a di Pisa\\
Largo B. Pontecorvo 5, 56127 Pisa - ITALY\\
{\tt giuseppe.buttazzo@unipi.it}\\
{\tt http://www.dm.unipi.it/pages/buttazzo/}

\bigskip\noindent
Francesco Paolo Maiale:
Scuola Normale Superiore\\
Piazza dei Cavalieri 7, 56126 Pisa - ITALY\\
{\tt francesco.maiale@sns.it}\\
{\tt https://poisson.phc.dm.unipi.it/\char`\~fpmaiale/}

\end{document}